\documentclass[a4paper,reqno]{amsart}

\usepackage{amsmath,amssymb,amsthm}
\usepackage{graphicx}
\usepackage{cite}
\usepackage{xcolor}
\usepackage{mathtools}
\usepackage{enumerate}
\usepackage{nicefrac}
\usepackage{bm}
\usepackage{a4wide}
\usepackage{algpseudocode}
\usepackage{booktabs}
\usepackage{refcheck}

\newcommand{\dif}{\,\text{d}}

\renewcommand{\phi}{\varphi}

\newcommand{\timejump}[1]{\lbrack\!\lbrack#1\rbrack\!\rbrack} 
\makeatletter
\newcommand*\bigcdot{\mathpalette\bigcdot@{.5}}
\newcommand*\bigcdot@[2]{\mathbin{\vcenter{\hbox{\scalebox{#2}{$\m@th#1\bullet$}}}}}
\makeatother
\newcommand{\tightoverset}[2]{%
  \mathop{#2}\limits^{\vbox to -.45ex{\kern-0.75ex\hbox{$#1$}\vss}}}
\makeatletter
\DeclareRobustCommand\widecheck[1]{{\mathpalette\@widecheck{#1}}}
\def\@widecheck#1#2{%
    \setbox\z@\hbox{\m@th$#1#2$}%
    \setbox\tw@\hbox{\m@th$#1%
       \widehat{%
          \vrule\@width\z@\@height\ht\z@
          \vrule\@height\z@\@width\wd\z@}$}%
    \dp\tw@-\ht\z@
    \@tempdima\ht\z@ \advance\@tempdima2\ht\tw@ \divide\@tempdima\thr@@
    \setbox\tw@\hbox{%
       \raise\@tempdima\hbox{\scalebox{1}[-1]{\lower\@tempdima\box
\tw@}}}%
    {\ooalign{\box\tw@ \cr \box\z@}}}
\makeatother

\newtheorem{Remark}[equation]{Remark}

\newtheorem{theorem}[equation]{Theorem}

\newtheorem{lemma}[equation]{Lemma}

\numberwithin{equation}{section}

\title[A quasi-optimal test norm for a DPG discretization of convection-diffusion]{A robust quasi-optimal test norm for a DPG discretization of the convection-diffusion equation}
\author{Stephen Metcalfe}
\address{Department of Mechanical Engineering, McGill University, Montr{\'e}al, H3A 0C3, Canada}
\email{smetcalfephd@gmail.com}
\author{Siva Nadarajah}
\address{Department of Mechanical Engineering, McGill University, Montr{\'e}al, H3A 0C3, Canada}
\email{siva.nadarajah@mcgill.ca}

\begin{document}

\begin{abstract}

In this work, we propose a new quasi-optimal test norm for a discontinuous Petrov-Galerkin (DPG) discretization of the ultra-weak formulation of the convection-diffusion equation. We prove theoretically that the proposed test norm leads to bounds between the target norm and the energy norm induced by the test norm which are robust with respect to the diffusion parameter in the solution and gradient components and have favorable scalings in the trace components. We conclude with numerical experiments to confirm our theoretical results.

\end{abstract}

\keywords{Discontinuous Petrov-Galerkin (DPG), Convection-Diffusion, Robust Test Norm}

\subjclass[2010]{65N12, 65N22}

\maketitle

\section{Introduction}

Let $\displaystyle \Omega \subset \mathbb{R}^d$ be a bounded polyhedral domain and consider the problem
\begin{equation}\label{model_primal}
\begin{aligned}
\nabla \bigcdot ({\bf a}u - \varepsilon \nabla u) &= f \qquad && \text{in }  \Omega,  \\ u &=0 \mbox{ } && \text{on }  \Gamma_D,  \\ ({\bf a}u - \varepsilon \nabla u) \bigcdot \bm{n} &=0 \mbox{ } && \text{on }  \Gamma_N,
\end{aligned}
\end{equation}
with $\bm{n}$ denoting the outward unit normal to the boundary $\partial\Omega$ and where $\Gamma_D$ and $\Gamma_N$ are two disjoint open subsets satisfying $\overline{\Gamma}_D \cup \overline{\Gamma}_N = \partial\Omega$. Here, the data satisfies $\varepsilon \in \mathbb{R}^+$, ${\bf a} \in [C(\bar{\Omega})]^d$ with $C(\bar{\Omega}) \ni \nabla \bigcdot {\bf a} \geq 0$ and $f \in H^{-1}(\Omega)$. 

The solution to problem \eqref{model_primal} is difficult to approximate numerically due to the presence of \emph{boundary/interior layers} {\bf--} areas of steep solution gradients with a width dependent upon $\varepsilon$. It is well known that regardless of what formulation of problem \eqref{model_primal} is used that discretizing it using the standard finite element method results in spurious oscillations occuring in the numerical solution near the layers until a sufficient number of elements have been added locally with the problem worsening as $\varepsilon \to 0^+$; for this reason, multiple stable schemes have been developed for problem \eqref{model_primal} over the years. The discontinuous Galerkin (DG) method originally developed by Reed/Hill for the neutron transport equation \cite{RH73} has been successfully applied to the convection-diffusion equation through a variety of different stable discretizations \cite{A82, BO99, DSW04, RWG99}. The so-called hybridizable discontinuous Galerkin (HDG) methods are another popular class of schemes used to discretize the convection-diffusion equation, cf., \cite{QS16, CDGRS09, NPC09, ES10, O14}. We also cannot discuss numerical methods for convection-diffusion equations without mentioning the highly successful streamline-upwind Petrov-Galerkin (SUPG) method \cite{BH82,BH79}; the SUPG method is unique among all of these previously mentioned methods in the sense that it uses standard conforming finite elements for the trial space but uses a special space of test functions, biased in the upwind direction, to impart the method its stability \cite{C13}. This idea of test functions imparting stability to numerical methods in fact dates back to \cite{HHZM77} and it is this critical observation that lies at the heart of the discontinuous Petrov-Galerkin method.

The discontinuous Petrov-Galerkin (DPG) method was originally introduced by Demkowicz/Gopalakrishnan in the context of the transport equation \cite{DG10}. This idea was then abstracted and applied to a variety of different equations \cite{ZMDGPC11, DG11b, DGN12}. A thorough analysis of the DPG method, as applied to the Poisson equation, also exists \cite{DG11}. The DPG method can be applied to the abstract variational formulation of finding $u \in U$ such that
\begin{equation}
\begin{aligned}\label{contproblem}
B(u,v) = l(v) \qquad \forall v \in V,
\end{aligned}
\end{equation}
where $U$ and $V$ are some Hilbert Spaces, $B(\cdot,\cdot):U\times V \to \mathbb{R}$ is a continuous bilinear form and $l:V \to \mathbb{R}$ is a continuous linear functional. As is common, for the discretization of \eqref{contproblem}, we seek $u_h \in U_h \subset U$ such that
\begin{equation}
\begin{aligned}\label{discreteproblem}
B(u_h,v_h) = l(v_h) \qquad \forall v_h \in V_h \subset V,
\end{aligned}
\end{equation}
$\dim(U_h) = \dim(V_h) < \infty$. In general, the existence of a solution to the continuous problem \eqref{contproblem} does not guarantee that the discrete problem \eqref{discreteproblem} also has a solution \emph{unless} $V_h$ is chosen to be the so-called space of optimal test functions. To be specific, \eqref{discreteproblem} becomes the \emph{(theoretical) DPG method} when the test space is chosen to be the \emph{space of optimal test functions} $V_h := T(U_h)$ where $T:U \to V, u \mapsto Tu$ is the \emph{trial to test operator} given by the unique solution of the variational problem
\begin{equation}
\begin{aligned}
\label{trialtotest}
(Tu, v)_V = B(u, v) \qquad \forall v \in V.
\end{aligned}
\end{equation}
Here, the inner product $(\cdot,\cdot)_V$ induces a norm $||\cdot||_V$ referred to as the \emph{test norm} which, effectively, defines the DPG method via the variational equation above.  Note that the use of this test space, in addition to guaranteeing the solvability of \eqref{discreteproblem}, also means that $u_h \in U_h$ solving \eqref{discreteproblem} gives the best approximation error in the \emph{energy norm} $||\cdot||_E$, i.e.,
\begin{equation}
\begin{aligned}\label{energynorm}
||u-u_h||_E := \sup_{v \in V} \frac{B(u-u_h,v)}{||v||_V} = \inf_{w_h \in U_h}||u-w_h||_E.
\end{aligned}
\end{equation}

We now remark that the space of optimal test functions, as defined above, cannot actually be computed since \eqref{trialtotest} is an infinite dimensional problem; therefore, for the \emph{practical DPG method} we must approximate this variational equation. Typically, this is done through an enriched test space $\widetilde{V}_h$ based on the same mesh as $U_h$ but with higher polynomial degree so that $\dim(\widetilde{V}_h) > \dim(U_h)$ though other methods of forming this enriched test space do exist \cite{SMD19, NCC11}. Thus, if we let $\{\varphi_i\}$ denote a basis for $U_h$ then we are seeking a basis $\{\psi_i \in \widetilde{V}_h\}$ defining $V_h$ such that
\begin{equation}
\begin{aligned}
\label{practicaltrialtotest}
(\psi_i,\tilde{v}_h)_V = B(\varphi_i,\tilde{v}_h) \qquad \forall \tilde{v}_h \in \widetilde{V}_h.
\end{aligned}
\end{equation}
However, this still presents a problem; namely, \eqref{practicaltrialtotest} results in a matrix-vector system of higher dimension than that of the original discretization! This difficulty can be overcome by using a discontinuous enriched test space $\widetilde{V}_h$ and, hence, a discontinuous test space $V_h$. Under the additional assumption that the test norm $||\cdot||_V$ is localizable then \eqref{practicaltrialtotest} can now be solved elementwise instead of globally yielding a practical method.  Note that $V_h$ must be a subspace of $V$ which means that $V$ must also consist of discontinuous functions; it is necessary to take this into consideration when formulating the bilinear form.

As noted above, the DPG method converges optimally in the energy norm $||\cdot||_E$, however, we usually prefer to measure the error in a norm of our choice $||\cdot||_U$ appropriate to the problem at hand. By duality, $||\cdot||_U$ defines a corresponding norm $||\cdot||_{V,\, \text{opt}}$ on $V$ referred to as the \emph{optimal test norm}, viz.,
\begin{equation}
\begin{aligned}\label{optimaltestnorm}
||{v}||_{V,\, \text{opt}} := \sup_{{u} \in U} \frac{B({u}, {v})}{||{u}||_U}.
\end{aligned}
\end{equation}
The discrepency between the use of $||\cdot||_V$ in the DPG method and the `optimal norm' $||\cdot||_{V,\, \text{opt}}$ results in a worse convergence bound formulated precisely in the following convergence result from Theorem 2.1 in \cite{ZMDGPC11}.
\begin{theorem}\label{DPGconvg}
Let ${u}$ and ${u}_h$ denote the solutions of \eqref{contproblem} and \eqref{discreteproblem}, respectively and suppose that the practical norm on $V$ is equivalent to the optimal test norm on V, i.e., that there exists positive constants $C_L$ and $C_U$ such that 
$$C_L||{v}||_V \leq ||{v}||_{V, \text{opt}} \leq C_U||{v}||_V \qquad \forall {v} \in V.$$
Further suppose that the bilinear form $B$ satisfies
$$B({u}, {v}) = 0 \quad \forall {v} \in V \implies {u} = {0}.$$ 
Then
$$||{u} - {u}_h||_U \leq \frac{C_U}{C_L}\inf_{{w}_h \in U_h} ||{u} - {w}_h||_U.$$
\end{theorem}

Theorem \ref{DPGconvg} implies that we want to choose the test norm $||\cdot||_V$ to be as close as we can practically get to the optimal test norm $||\cdot||_{V,\, \text{opt}}$ in order to force $C_L$ and $C_U$ to be as near to unity as possible. Note that we cannot use $||\cdot||_{V,\, \text{opt}}$  itself in the DPG method since, as we will see below, for practical choices of the norm $||\cdot||_U$ the optimal test norm contains `jump norms' which couple nearby elements meaning it is non-localizable. In the context of convection-diffusion, then, the goal is to find an appropriate trial norm $||\cdot||_U$  in which to measure the error along with a test norm $||\cdot||_V$ which is as close as possible to the optimal test norm thus resulting in a good ratio of the constants $C_U$ to $C_L$.

Most of the DPG discretizations for the convection-diffusion equation are based on the ultra-weak formulation (to be derived in the next section) for which we have a vector-valued test function $\bm{v} = (v,\bm{\tau})$ with $v$ corresponding to $u$ and $\bm{\tau}$ corresponding to $\nabla u$. Various different test norms have been proposed for the DPG method as applied to the convection-diffusion equation \eqref{model_primal}. In \cite{DH13}, multiple test norms are trialed including the \emph{quasi-optimal test norm}
\begin{equation}
\begin{aligned}
||\bm{v}||_{V,\,\text{QO}}  := \sqrt{||\nabla \bigcdot \bm{\tau} - {\bf a} \bigcdot  \nabla v||^2+||\varepsilon ^{-1}\bm{\tau} + \nabla v||^2 + ||v||^2},
\end{aligned}
\end{equation}
which has performed well in the past on many PDE problems of very different character. Results show that this norm performs well up to a point using the standard $p$-enriched test space approach but that for higher polynomial degrees it was necessary to use computationally expensive Shishkin-type meshes for the enriched test space \cite{DH13}. This problem is somewhat remedied in  \cite{C13, CHBTD14} through the use of the \emph{mesh-dependent test norm} 
\begin{equation}
\begin{aligned}
||\bm{v}||_{V,\, \text{MD}} := \sqrt{||C_v v||^2 + \varepsilon||\nabla v||^2 + ||{\bf a} \bigcdot \nabla v||^2 + ||C_{\bm{\tau}} \bm{\tau}||^2 + ||\nabla \bigcdot \bm{\tau}||^2},
\end{aligned}
\end{equation}
where $\displaystyle C_v |_K := \min\{\sqrt{\varepsilon/|K|}, 1\}$, $C_{\bm{\tau}} |_K := \min\{1/\sqrt{\varepsilon}, 1/\sqrt{|K|}\}$, $K \in \mathcal{T}$ where $\mathcal{T}$ is some mesh of $\Omega$. It was shown that optimal test functions computed under this norm do not have boundary layers allowing for the standard $p$-enriched test space approach to work for this norm. All of the works \cite{C13, CHBTD14, DH13} contain upper and lower bounds between $||\bm{u}||_U$ and $||\bm{u}||_E$ with which we will compare bounds for our proposed test norm later in this paper. Another work \cite{NCC13} shows that the \emph{quasi-optimal test norm} 
\begin{equation}
\begin{aligned}||\bm{v}||_{V, \,\text{QO}_2}  := \sqrt{||\nabla \bigcdot \bm{\tau} - {\bf a} \bigcdot  \nabla v||^2+||\varepsilon^{-1}\bm{\tau} + \nabla v||^2 + \alpha_1||v||^2 + \alpha_2||\bm{\tau}||^2},
\end{aligned}
\end{equation}
performs well with the authors choosing the values $\alpha_1 = 1$, $\alpha_2 = \varepsilon^{-3/2}$ for these parameters in their numerical experiments. As in \cite{DH13}, Shishkin-type meshes are used for the enriched test space in that work.

Here, we propose the \emph{mesh-dependent quasi-optimal test norm}
\begin{equation}
\begin{aligned}\label{ourtestnorm}
||\bm{v}||_{V}  := \sqrt{\varepsilon||\nabla \bigcdot \bm{\tau} - {\bf a} \bigcdot  \nabla v||^2+||C_{\bm{\tau}}(\bm{\tau} + \varepsilon \nabla v)||^2 + \varepsilon||v||^2 + \varepsilon||\nabla v||^2},
\end{aligned}
\end{equation}
for use in the DPG method approximating \eqref{model_primal}. We prove upper and lower bounds between $||\bm{u}||_U$ and $||\bm{u}||_E$ for $||\cdot||_V$ that are robust in both of the field variables and have favorable scalings in the trace components with minimal assumptions on the convection (we say some norm is \emph{robust (with respect to $\varepsilon$)} if the ratio of upper to lower bounds is independent of $\varepsilon$). Similar upper/lower bounds can also be proven for our proposed test norm $||\cdot||_V$ but with the term $\varepsilon||\nabla v||^2$ replaced by either $\varepsilon^{-1}||\bm{\tau}||^2$ or $\varepsilon||\nabla \bigcdot \bm{\tau}||^2$ just with differing ($\varepsilon$-independent) constants. In light of this observation, we note the similarities between our proposed test norm $||\cdot||_V$ and the quasi-optimal test norm $||\cdot||_{V, \,\text{QO}_2}$ of \cite{NCC13}; due to this, we restrict ourselves to a comparison between the test norms $||\cdot||_V$ and $||\cdot||_{V,\,\text{MD}}$ only in the numerical experiments section.

The remainder of this paper is organized as follows: in the next section, we derive the ultra-weak variational formulation for \eqref{model_primal} and in section 3 we introduce the function spaces necessary in order to properly define the ultra-weak variational formulation. In section 4, we prove upper and lower bounds between $||\bm{u}||_U$ and $||\bm{u}||_E$ for the proposed test norm \eqref{ourtestnorm} and compare our results with similar results already in the literature. We then apply our test norm in a variety of numerical experiments in section 5 with a special emphasis on a comparison with the test norm $||\cdot||_{V,\,\text{MD}}$ from \cite{C13, CHBTD14} before drawing conclusions in section 6.

\section{Ultra-Weak Formulation}

Setting $\bm{\sigma} = \nabla u$ in \eqref{model_primal} gives the \emph{mixed formulation}
\begin{equation}\label{model_mixed}
\begin{aligned}
\nabla \bigcdot ({\bf a}u - \varepsilon \bm{\sigma}) &= f \qquad && \text{in }  \Omega,  \\
\bm{\sigma} - \nabla u &= 0 \qquad && \text{in }  \Omega,  \\
 u &=0 \mbox{ } && \text{on }  \Gamma_D,  \\  ({\bf a}u - \varepsilon\bm{\sigma}) \cdot \bm{n}&=0 \mbox{ } && \text{on }  \Gamma_N,
\end{aligned}
\end{equation}
Next, we assume that $\mathcal{T} = \{ K \}$ is a triangulation of the domain $\Omega$ with $K$ denoting a generic open element of diameter $h_K$ with which we associate the mesh skeleton $\Gamma$, that is, the union of all edges $E$ in the triangulation. If we let $v$ and $\bm{\tau}$ denote scalar and vector (respectively) valued test functions which are allowed to be discontinuous across $\Gamma$ then multiplying the first equation by $v$ and the second equation by $\bm{\tau}$ and integrating over $K \in \mathcal{T}$ we obtain
\begin{equation}
\begin{aligned}
\int_K \nabla \bigcdot ({\bf a}u - \varepsilon \bm{\sigma})v \dif x &= \int_K  fv \dif x,  \\
\int_K (\bm{\sigma} - \nabla u) \bigcdot \bm{\tau} \dif x &= 0.
\end{aligned}
\end{equation}
Using integration by parts to pass the trial function derivatives over to the test functions yields
\begin{equation}
\begin{aligned}
\int_K (\varepsilon \bm{\sigma} - {\bf a}u) \bigcdot  \nabla v \dif x + \int_{\partial K} ({\bf a}uv - \varepsilon \bm{\sigma}v) \bigcdot \bm{n}_K \dif s&= \int_K  fv \dif x,  \\
\int_K \{u \nabla \bigcdot \bm{\tau}+\bm{\sigma} \bigcdot \bm{\tau} \} \dif x - \int_{\partial K} u\,\bm{\tau} \bigcdot \bm{n}_K \dif s&= 0,
\end{aligned}
\end{equation}
with $\bm{n}_K$ denoting the outward unit normal to $\partial K$. Adding the two equations gives
\begin{equation}
\begin{aligned}
\int_K \{u (\nabla \bigcdot \bm{\tau} - {\bf a} \bigcdot  \nabla v )+\bm{\sigma} \bigcdot (\bm{\tau} + \varepsilon \nabla v) \} \dif x  +\int_{\partial K}  ({\bf a}uv-u\bm{\tau} - \varepsilon \bm{\sigma}v) \bigcdot \bm{n}_K \dif s&= \int_K  fv \dif x.
\end{aligned}
\end{equation}
If $E \subset \partial K$ is some edge then upon setting $\hat{u} |_{E} := u|_{E}$ and $\hat{\sigma}_n |_{E} := ({\bf a}u - \varepsilon \bm{\sigma}) \bigcdot \bm{n}  |_{E}$ where $\bm{n}$ is some unit normal vector to $E$ (or specifically the outward unit normal if $E \subset \partial \Omega$) we obtain 
\begin{equation}
\begin{aligned}
B_K(\{u,\bm{\sigma},\hat{u},\hat{\sigma}_n\}, \{v, \bm{\tau}\}) & := \int_K \{u (\nabla \bigcdot \bm{\tau} - {\bf a} \bigcdot  \nabla v )+\bm{\sigma} \bigcdot (\bm{\tau} + \varepsilon \nabla v) \} \dif x  - \int_{\partial K} \hat{u} \, \bm{\tau} \bigcdot \bm{n}_K \dif s 
\\&+\int_{\partial K} \text{sgn}(\bm{n}_K)\,\hat{\sigma}_n v \dif s = \int_K  fv \dif x =: (f,v)_K,
\end{aligned}
\end{equation}
where for some edge $E \subset \partial K$ we have 
\begin{equation}
\begin{aligned}
\text{sgn}(\bm{n}_K) := \begin{cases} 1 & \text{if } \bm{n}_K = \bm{n}, \\ -1 & \text{if } \bm{n}_K = -\bm{n}.\end{cases}
\end{aligned}
\end{equation}
The \emph{ultra-weak formulation} is then obtained by summing over all elements $K$, viz.,
\begin{equation}
\begin{aligned}
\label{ultraweaklong}
B(\{u,\bm{\sigma},\hat{u},\hat{\sigma}_n\}, \{v, \bm{\tau}\}) := \sum_{K \in \mathcal{T}} B_K (\{u,\bm{\sigma},\hat{u},\hat{\sigma}_n\}, \{v, \bm{\tau}\}) = \sum_{K \in \mathcal{T}} (f,v)_K =: (f,v).
\end{aligned}
\end{equation}
For brevity, it is useful to define the group variables $\bm{u} := (u,\bm{\sigma},\hat{u},\hat{\sigma}_n)$ and $\bm{v} := (v, \bm{\tau})$ thus allowing us to rewrite problem \eqref{ultraweaklong} more compactly as finding $\bm{u} \in U$ such that
\begin{equation}
\begin{aligned}
\label{ultraweakshort}
B(\bm{u},\bm{v}) = (f,v) \qquad \forall \bm{v} \in V,
\end{aligned}
\end{equation}
with the vector spaces $U$ and $V$ to be defined in the next section. Additionally, the bilinear form $B$ may be further simplified to
\begin{equation}
\begin{aligned}
\label{bilinearform}
B(\bm{u},\bm{v})  = \sum_{K \in \mathcal{T}} \int_K \{u (\nabla \bigcdot \bm{\tau} - {\bf a} \bigcdot  \nabla v )+\bm{\sigma} \bigcdot (\bm{\tau} + \varepsilon \nabla v) \} \dif x + \int_{\Gamma} \{\hat{\sigma}_n\timejump{v}-\hat{u}\timejump{\bm{\tau}}\} \dif s,
\end{aligned}
\end{equation}
by taking advantage of jump notation
\begin{equation}
\begin{aligned}
\notag
&\timejump{v}\big|_E := \text{sgn}(\bm{n}_K)v|_K + \text{sgn}(\bm{n}_{K'})v|_{K'}, \quad &&\timejump{\bm{\tau}} \big|_E := \mathbf{\tau}|_K \bigcdot \mathbf{n}_K + \mathbf{\tau}|_{K'} \bigcdot \mathbf{n}_{K'} \quad  && E = \partial K \cap \partial K', \\
&\timejump{v}\big|_E := v|_E, &&\timejump{\bm{\tau}} \big|_E := \mathbf{\tau}|_E \bigcdot \mathbf{n}, && E \subset \partial\Omega.
\end{aligned}
\end{equation}

\section{Function Spaces}

In order for the ultra-weak formulation \eqref{ultraweakshort} to be well-defined, the vector spaces $U$ and $V$ need to be specified. Hence, this section is dedicated to introducing the notation necessary to define these spaces as well as some miscellaneous definitions and theorems which will be needed in the forthcoming analysis.

We begin by defining the broken $L^2$-norm
$$||\bm{\sigma}|| := \bigg(\sum_{K \in \mathcal{T}} \int_K |\bm{\sigma}(x)|^2 \dif x \bigg)^{\!1/2},$$
for the space $[L^2(\Omega)]^n$, $n \in \mathbb{N}$. Next, we introduce the $\varepsilon$-scaled broken $H^1$-norm
$$||u||_{H^1_{\varepsilon}(\Omega; \,\mathcal{T})} := \sqrt{\varepsilon^{-1}||u||^2 + \varepsilon||\nabla u||^2},$$
for the space $H^1(\Omega; \,\mathcal{T}) := \{u:\Omega \to \mathbb{R} \, | \, u \in H^1(K) \,\,\forall K \in \mathcal{T} \}$ and the $\varepsilon$-scaled broken $H(\text{div})$-norm
$$||\bm{\sigma}||_{H_{\varepsilon}(\text{div}, \,\Omega; \,\mathcal{T})} :=  \sqrt{\varepsilon^{-1}||\bm{\sigma}||^2 + \varepsilon^{-1}||\nabla \bigcdot \bm{\sigma}||^2},$$
for the space $H(\text{div}, \,\Omega; \,\mathcal{T}) := \{\bm{\sigma}:\Omega \to \mathbb{R}^d \, | \, \bm{\sigma} \in H(\text{div},K) \,\, \forall K \in \mathcal{T}  \}$. We remark that {\bf these norms are non-standard} and will be used in this way for the remainder of the paper. Additionally, we will also require the use of the standard Sobolev spaces
$H^1_D := \{u \in H^1(\Omega) \,|\, u = 0 \text{ on } \Gamma_D\}$ and $H_N(\text{div}, \,\Omega) := \{ \bm{\sigma} \in H(\text{div},\,\Omega) \, | \, \bm{\sigma} \bigcdot \bm{n} = 0 \text{ on } \Gamma_N\}$ where the boundary equalities are to be understood in the sense of traces. 

We also need to introduce spaces in order to be able to describe the trace variables. To that end, we set
\begin{equation}
\begin{aligned}
\notag
{H_D^{1/2}(\Gamma)} &:= \{\hat{u}:\Gamma \to \mathbb{R} \, | \, \exists w \in H^1_D(\Omega) \text{ such that } \hat{u} = w|_{\Gamma} \}, \\
{H_N^{-1/2}(\Gamma)} &:= \{\hat{\sigma}_n:\Gamma \to \mathbb{R} \, | \, \exists \bm{w} \in H_N(\text{div}, \,\Omega) \text{ such that } \hat{\sigma}_n = \bm{w} \bigcdot \bm{n} |_{\Gamma} \},
\end{aligned}
\end{equation}
where these boundary equalities are again to be understood in the sense of traces. These trace spaces admit so-called minimal extension norms
\begin{equation}
\begin{aligned}\notag
||\hat{u}||_{H^{1/2}_D(\Gamma)} & := \inf\{||w||_{H_{\varepsilon}^1(\Omega; \, \mathcal{T})} \, | \, w \in H^1_D(\Omega) \text{ and } \hat{u}  = w |_{\Gamma} \}, \\
||\hat{\sigma}_n||_{H^{-1/2}_N(\Gamma)} & := \inf\{||\bm{w}||_{H_{\varepsilon}(\text{div},\Omega; \, \mathcal{T})} \, | \, \bm{w} \in H_N(\text{div},\Omega) \text{ and } \hat{\sigma}_n  = \bm{w} \bigcdot \bm{n} |_{\Gamma} \}.
\end{aligned}
\end{equation}

\smallskip

With this notation at hand, we are now ready to state the vector spaces for problem \eqref{ultraweakshort}. We are seeking $\bm{u} \in U := L^2(\Omega) \!\times\! [L^2(\Omega)]^d \!\times\! H_D^{1/2}(\Gamma) \!\times\! H_N^{-1/2}(\Gamma)$ such that \eqref{ultraweakshort} holds for all $\bm{v} \in V:=H^1(\Omega; \mathcal{T}) \! \times \! H(\text{div}, \,\Omega; \,\mathcal{T})$. To the trial space $U$, we associate the $\varepsilon$-scaled norm
\begin{equation}
\begin{aligned}\label{unorm}
||\bm{u}||_U := \sqrt{\varepsilon^{-1}||u||^2 + \varepsilon||\bm{\sigma}||^2 + ||\hat{u}||_{H_D^{1/2}(\Gamma)}^2 + ||\hat{\sigma}_n||_{H_N^{-1/2}(\Gamma)}^2},
\end{aligned}
\end{equation}
while we recall that the test space $V$ can be endowed with the so-called optimal test norm \eqref{optimaltestnorm}. It can be shown, cf. \cite{C13, CHBTD14}, that the optimal test norm \eqref{optimaltestnorm} corresponding to the bilinear form \eqref{bilinearform} with respect to the norm $||\cdot||_U$ is given by
\begin{equation}
\begin{aligned}
\label{actualoptimaltestnorm}
||\bm{v}||_{V,\, \text{opt}}  &= \sqrt{\varepsilon||\nabla \bigcdot \bm{\tau} - {\bf a} \bigcdot  \nabla v||^2+\varepsilon^{-1}||\bm{\tau} + \varepsilon \nabla v||^2 + ||\timejump{v}||^2_{\Gamma} + ||\timejump{\bm{\tau}}||_{\Gamma}^2},
\end{aligned}
\end{equation}
where
\begin{equation}
\begin{aligned}\notag
||\timejump{v}||_{\Gamma} & :=  \sup_{\bm{w} \in H_N(\text{div}, \,\Omega)} \frac{\int_{\Gamma}  \bm{w} \bigcdot \bm{n}\, \timejump{v}\dif s}{||\bm{w}||_{H_{\varepsilon}(\text{div},\Omega; \, \mathcal{T})}}, \\
||\timejump{\bm{\tau}}||_{\Gamma}  &:= \sup_{w \in H^1_D(\Omega)} \frac{\int_{\Gamma} w\timejump{\bm{\tau}}\dif s}{||w||_{H_{\varepsilon}^1(\Omega; \, \mathcal{T})}}.
\end{aligned}
\end{equation}

Ideally, we would use \eqref{actualoptimaltestnorm} in our DPG method, however, the jump norms couple nearby elements together meaning a global instead of local solve would be required in order to calculate the optimal test functions required for the DPG method to function; it is therefore practically necessary to find a localizable upper bound for \eqref{actualoptimaltestnorm}. 

In the forthcoming analysis, we will make extensive use of the following theorem which is the well-known {\em classical Poincar{\'e}-Friedrichs inequality} the proof of which can be readily found in many functional analysis textbooks.

\begin{theorem}\label{pf}
For any $v \in H^1_D(\Omega)$ we have
$$||v|| \leq C_P||\nabla v||,$$
where $C_P$ is some positive constant which is only dependent upon the size of the domain $\Omega$.
\end{theorem}

For DPG analysis, the classical Poincar{\'e}-Friedrichs inequality is not generally sufficient given we are usually dealing with discontinuous functions. To that end, we have the following discontinuous variant of this classical theorem: the {\em broken Poincar{\'e}-Friedrichs inequality}.
\begin{theorem}\label{brokenpf}
For any $v \in H^1(\Omega; \,\mathcal{T})$ we have
$$||v|| \leq C_P ||\nabla v|| + \varepsilon^{-1/2}\sqrt{1+C_P^2}\,||\timejump{v}||_{\Gamma},$$
where $C_P$ is the classical Poincar{\'e}-Friedrichs constant.
\end{theorem}
\begin{proof}
Let $\phi \in \{w \in H^1(\Omega) \, | \, w= 0 \text{ on } \Gamma_D \text{ and }\nabla w \bigcdot \bm{n} = 0 \text{ on } \Gamma_N\}$ be the solution of the problem $-\Delta \phi = v$ then we have stability, viz.,
\begin{equation}
\begin{aligned}
\notag||\nabla \phi|| \leq C_P||v||.
\end{aligned}
\end{equation}
Next, observe that
\begin{equation}
\begin{aligned}
\notag||v||^2 &= -\int_{\Omega} v\Delta\phi \dif x = \sum_{K \in \mathcal{T}} \left\{\int_K \nabla v \bigcdot \nabla \phi \dif x - \int_{\partial K} v\nabla\phi \bigcdot \bm{n} \dif s \right\}.
\end{aligned}
\end{equation}
Now since
\begin{equation}
\begin{aligned}\notag
- \sum_{K \in \mathcal{T}} \int_{\partial K} v\nabla\phi \bigcdot \bm{n} \dif s = -\frac{\int_\Gamma \nabla \phi \bigcdot \bm{n}\, \timejump{v}\dif s}{||\nabla \phi||_{H_{\varepsilon}(\text{div},\Omega; \, \mathcal{T})}}||\nabla \phi||_{H_{\varepsilon}(\text{div},\Omega; \, \mathcal{T})} \leq  ||\timejump{v}||_{\Gamma}||\nabla \phi||_{H_{\varepsilon}(\text{div},\Omega; \, \mathcal{T})},
\end{aligned}
\end{equation}
then with the aid of the Cauchy-Schwarz inequality we have
\begin{equation}
\begin{aligned}
\notag
\notag||v||^2 \leq ||\nabla v||||\nabla \phi|| + ||\timejump{v}||_{\Gamma}||\nabla \phi||_{H_{\varepsilon}(\text{div},\Omega; \, \mathcal{T})}.
\end{aligned}
\end{equation}
The result then follows from the definition of $||\cdot||_{H_{\varepsilon}(\text{div},\Omega; \, \mathcal{T})}$ and the stability property.
\end{proof}

\section{Robust Test Norm}

In this section, we will work to show that the norm
$$||\bm{v}||_{V}  := \sqrt{\varepsilon||\nabla \bigcdot \bm{\tau} - {\bf a} \bigcdot  \nabla v||^2+\varepsilon^{-1}||\bm{\tau} + \varepsilon \nabla v||^2 + \varepsilon||v||^2 + \varepsilon||\nabla v||^2},$$
is a test norm sufficiently close to $||\cdot||_{V,\,\text{opt}}$ that the corresponding \emph{energy norm} \eqref{energynorm} on $U$ generated by $||\cdot||_V$ is both a robust upper and lower bound with respect to $\varepsilon$ in the $u$ and $\bm{\sigma}$ solution components while also having favorable scaling with respect to $\varepsilon$ in the $\hat{u}$ and $\hat{\sigma}_n$ components.

\subsection{Upper Bound}

We begin by proving that $||\cdot||_U$ is a robust upper bound for $||\cdot||_E$. In order to do this, we first need to show that $||\cdot||_V$ is an upper bound for $||\cdot||_{V,\,\text{opt}}$. This is the content of the next theorem.

\begin{theorem}\label{Voptupperbound}
For any $\bm{v}\in V$ we have
$$||\bm{v}||_{V,\, \text{opt}} \leq C_U||\bm{v}||_V,$$
where $C_U := \sqrt{3 + 2||{\bf a}||_{L^\infty(\Omega)}^2}$.
\end{theorem}
\begin{proof}
Since the field terms are already present in the correct form, we only need to work on bounding the jump norms. Through integration by parts, the Cauchy-Schwarz inequality and the elementary inequality $ab + cd \leq \sqrt{a^2+c^2}\sqrt{b^2+d^2}$ we have
\begin{equation}
\begin{aligned}\notag
||\timejump{v}||_{\Gamma} & =  \sup_{\bm{w} \in H_N(\text{div}, \,\Omega)} \frac{\sum_{K \in \mathcal{T}} \int_K (v \nabla \bigcdot \bm{w}+\bm{w} \bigcdot \nabla v ) \dif x   }{||\bm{w}||_{H_{\varepsilon}(\text{div},\Omega; \, \mathcal{T})}} \\
& \leq \sup_{\bm{w} \in H_N(\text{div}, \,\Omega)} \frac{||v||||\nabla \bigcdot \bm{w}||  +  ||\nabla v||||\bm{w}||}{||\bm{w}||_{H_{\varepsilon}(\text{div},\Omega; \, \mathcal{T})}}  \\
& \leq \sup_{\bm{w} \in H_N(\text{div}, \,\Omega)} \frac{\sqrt{\varepsilon||v||^2 + \varepsilon||\nabla v||^2}\sqrt{\varepsilon^{-1}||\bm{w}||^2 + \varepsilon^{-1}||\nabla \bigcdot \bm{w}||^2}}{||\bm{w}||_{H_{\varepsilon}(\text{div},\Omega; \, \mathcal{T})}}  \\
& = \sqrt{\varepsilon||v||^2 + \varepsilon||\nabla v||^2}.
\end{aligned}
\end{equation}
Similarly,
\begin{equation}
\begin{aligned}\notag
||\timejump{\bm{\tau}}||_{\Gamma} & = \sup_{w \in H^1_D(\Omega)} \frac{\sum_{K \in \mathcal{T}} \int_K (\bm{\tau} \bigcdot \nabla w + w \nabla \bigcdot \bm{\tau}) \dif x   }{||w||_{H_{\varepsilon}^1(\Omega; \, \mathcal{T})}} \\
& \leq \sup_{w \in H^1_D(\Omega)} \frac{||\bm{\tau}||||\nabla w||  +  ||\nabla \bigcdot \bm{\tau}||||w|| }{||w||_{H_{\varepsilon}^1(\Omega; \, \mathcal{T})}} \\
& \leq \sup_{w \in H^1_D(\Omega)} \frac{\sqrt{\varepsilon^{-1}||\bm{\tau}||^2  +  \varepsilon||\nabla \bigcdot \bm{\tau}||^2}\sqrt{\varepsilon^{-1}||w||^2 + \varepsilon||\nabla w||^2} }{||w||_{H_{\varepsilon}^1(\Omega; \, \mathcal{T})}} \\
& = \sqrt{\varepsilon^{-1}||\bm{\tau}||^2  +  \varepsilon||\nabla \bigcdot \bm{\tau}||^2}.
\end{aligned}
\end{equation}
Thus we have
$$||\bm{v}||_{V,\, \text{opt}}^2 \leq \varepsilon||\nabla \bigcdot \bm{\tau} - {\bf a} \bigcdot  \nabla v||^2+\varepsilon^{-1}||\bm{\tau} + \varepsilon \nabla v||^2 + \varepsilon||v||^2 + \varepsilon||\nabla v||^2 + \varepsilon^{-1}||\bm{\tau}||^2 + \varepsilon||\nabla \bigcdot \bm{\tau}||^2.$$
All that remains is to show that this upper bound may in turn be bounded by $||{\bm v}||^2_V$. Using the triangle inequality, H{\"o}lder's inequality as well as the trivial inequality $(a+b)^2 \leq 2(a^2+b^2)$ yields
\begin{equation}
\begin{aligned}\notag
\varepsilon^{-1}||\bm{\tau}||^2 & \leq 2\varepsilon^{-1}||\bm{\tau} + \varepsilon\nabla v||^2 + 2\varepsilon||\nabla v||^2, \\
\varepsilon||\nabla \bigcdot \bm{\tau}||^2 & \leq 2\varepsilon||\nabla \bigcdot \bm{\tau} - {\bf a} \bigcdot  \nabla v||^2 + 2||{\bf a}||_{L^\infty(\Omega)}^2\varepsilon||\nabla v||^2,
\end{aligned}
\end{equation}
which upon substitution completes the proof.
\end{proof}

Having shown that $||\cdot||_V$ is an upper bound for $||\cdot||_{V,\,\text{opt}}$, proving that $||\cdot||_U$ is an upper bound for $||\cdot||_E$ is now trivial.
\begin{theorem}\label{unormlowerbound}
For any $\bm{u} \in U$ we have the upper bound 
$$||\bm{u}||_E \leq C_U||\bm{u}||_U,$$
where $C_U := \sqrt{3 + 2||{\bf a}||_{L^\infty(\Omega)}^2}$.
\end{theorem}
\begin{proof}
Theorem \ref{Voptupperbound} immediately implies that
$$||\bm{u}||_U = \sup_{\bm{v} \in V} \frac{B(\bm{u},\bm{v})}{||\bm{v}||_{V,\, \text{opt}}} \geq \frac{1}{C_U}\sup_{\bm{v} \in V} \frac{B(\bm{u},\bm{v})}{||\bm{v}||_{V}} = C_U^{-1}||\bm{u}||_E.$$
\end{proof}

\subsection{Lower Bound}

In order to find a lower bound for $||\bm{u}||_E$, we first need to derive an upper bound for $||\bm{v}||_{V}$; unfortunately, this is somewhat more involved than the upper bound for $||\bm{v}||_{V,\, \text{opt}}$. Here, we follow the approach taken in \cite{DG11} of decomposing the test function $\bm{v} = (v, \bm{\tau}) = (v_0, \bm{\tau}_0) + (v_1, \bm{\tau}_1) =: \bm{v}_0 + \bm{v}_1$ where $\bm{v_0} \in  H^1(\Omega; \,\mathcal{T})\!\times\!H(\text{div}, \,\Omega; \, \mathcal{T})$ satisfies
\begin{equation}
\begin{aligned}
\label{v0eqns}
\nabla \bigcdot \bm{\tau}_0 - {\bf a} \bigcdot  \nabla v_0 & = 0 \qquad \text{on } K, \\
\bm{\tau}_0 + \varepsilon \nabla v_0 &= 0 \qquad \text{on } K,
\end{aligned}
\end{equation}
while $\bm{v}_1 \in H^1_D(\Omega) \! \times \! H_N(\text{div}, \,\Omega)$ satisfies
\begin{equation}
\begin{aligned}
\label{v1eqns}
\nabla \bigcdot \bm{\tau}_1 - {\bf a} \bigcdot  \nabla v_1 & = \nabla \bigcdot \bm{\tau} - {\bf a} \bigcdot  \nabla v &\qquad \text{on } K, \\
\bm{\tau}_1 + \varepsilon \nabla v_1 &= \bm{\tau} + \varepsilon \nabla v &\qquad \text{on } K.
\end{aligned}
\end{equation}
The strategy is then to derive an upper bound for both $||\bm{v}_0||_V$ and $||\bm{v}_1||_V$. 

To that end, we first note that the field terms in $||\cdot||_V$ and $||\cdot||_{V,\,\text{opt}}$ are both identical thus we need only be concerned with bounding $\varepsilon||v_0||^2$, $\varepsilon||\nabla v_0||^2$, $\varepsilon||v_1||^2$ and  $\varepsilon||\nabla v_1||^2$. Additionally, the Poincar{\'e}-Friedrichs inequalities (Theorems \ref{pf} and \ref{brokenpf}) give us control over $\varepsilon||v_0||^2$ and $\varepsilon||v_1||^2$ provided we have bounds for $\varepsilon||\nabla v_0||^2$ and $\varepsilon||\nabla v_1||^2$, respectively. Therefore, in fact, we only need to bound these two gradient norms. We'll begin with the easier of the two norms: $||\bm{v}_1||_V^2$.
\begin{theorem}\label{v1bound}
We have the following upper bound for $\bm{v}_1$:
$$||\bm{v}_1||^2_V \leq \varepsilon||\nabla \bigcdot \bm{\tau} - {\bf a} \bigcdot  \nabla v||^2+(1/2)(5 + 3C_P^2)\varepsilon^{-1}||\bm{\tau} + \varepsilon \nabla v||^2,$$
where $C_P$ is the Poincar{\'e}-Friedrichs constant of Theorem \ref{pf}.
\end{theorem}
\begin{proof}
Recall that by design we have
\begin{equation}
\begin{aligned}\notag
\bm{\tau}_1 + \varepsilon \nabla v_1 &= \bm{\tau} + \varepsilon \nabla v &\qquad \text{on } K,
\end{aligned}
\end{equation}
$K \in \mathcal{T}$. Multiplying this by $\bm{\tau}_1 - \varepsilon\nabla v_1$, integrating over $K$ and summing over $K \in \mathcal{T}$ yields
\begin{equation}
\begin{aligned}\notag
||\bm{\tau}_1||^2 + \varepsilon^2||\nabla v_1||^2 &= (\bm{\tau} + \varepsilon \nabla v, \bm{\tau}_1 - \varepsilon\nabla v_1).
\end{aligned}
\end{equation}
Using the Cauchy-Schwarz inequality and the elementary inequality $2ab \leq \kappa\alpha^2 + \kappa^{-1}b^2$ gives
\begin{equation}
\begin{aligned}\notag
||\bm{\tau}_1||^2 + \varepsilon^2||\nabla v_1||^2&\leq (||\bm{\tau}_1||+\varepsilon||\nabla v_1||)||\bm{\tau} + \varepsilon \nabla v|| \\
&\leq ||\bm{\tau}_1||^2 + \frac{1}{4}||\bm{\tau} + \varepsilon \nabla v||^2 + \frac{\varepsilon^2}{2}||\nabla v_1||^2 + \frac{1}{2}||\bm{\tau} + \varepsilon \nabla v||^2.
\end{aligned}
\end{equation}
After rescaling we thus obtain that
\begin{equation}
\begin{aligned}\notag
\varepsilon||\nabla v_1||^2 \leq \frac{3}{2\varepsilon}||\bm{\tau} + \varepsilon \nabla v||^2.
\end{aligned}
\end{equation}
Since $v_1 \in H^1_0(\Omega)$, we can apply the classical Poincar{\'e}-Friedrichs inequality to obtain
$$\varepsilon||v_1||^2 \leq C_P^2\varepsilon||\nabla v_1||^2 \leq \frac{3C_P^2}{2\varepsilon}||\bm{\tau} + \varepsilon \nabla v||^2.$$
Combining the results for $\varepsilon||v_1||^2$ and $\varepsilon||\nabla v_1||^2$ yields the statement of the theorem.
\end{proof}
Next, we turn our attention to proving bounds for $\varepsilon||\nabla v_0||^2$ and, ultimately, $||\bm{v}_0||_V^2$.
To that end, we need the following Helmholtz-type decomposition of $\nabla v_0$:
\begin{equation}
\begin{aligned}\label{helm}
\nabla v_0 = \varepsilon \nabla \psi - {\bf a}\psi + \nabla \!\times \!z,
\end{aligned}
\end{equation}
where $\psi \in \{w \in H^1(\Omega) \, | \, w = 0 \text{ on } \Gamma_D \text{ and } (\varepsilon \nabla w - {\bf a}w) \bigcdot \bm{n}  = \nabla v_0 \bigcdot \bm{n} \text{ on } \Gamma_N \}$ and $z \in H(\text{curl}, \, \Omega)$. To see why this decomposition is useful, observe that
\begin{equation}
\begin{aligned}\label{gradv0bound}
||\nabla v_0||^2 &= \sum_{K \in \mathcal{T}} \int_{K} (\varepsilon \nabla \psi - {\bf a}\psi + \nabla \!\times \!z) \bigcdot \nabla v_0 \dif x  \\
&= \sum_{K \in \mathcal{T}} \int_{K} (-{\bf a}\psi \bigcdot \nabla v_0 - \bm{\tau}_0\bigcdot \nabla \psi) \dif x + \int_{\partial K} v_0\nabla \! \times \! z \bigcdot \bm{n}_K  \dif s \\
&= \sum_{K \in \mathcal{T}} \int_{K} \underbrace{(\nabla \bigcdot \bm{\tau}_0-{\bf a}\bigcdot \nabla v_0)}_{=0}\psi \dif x + \int_{\partial K} (v_0\nabla \! \times \! z \bigcdot \bm{n}_K - \bm{\tau}_0\psi \bigcdot \bm{n}_K) \dif s \\
&=\int_{\Gamma}(\timejump{v}\nabla \! \times \! z \bigcdot \bm{n}  - \timejump{\bm{\tau}}\psi ) \dif s \\
&\leq \varepsilon^{-1/2}||\timejump{v}||_{\Gamma}||\nabla \! \times \! z|| + ||\timejump{\bm{\tau}}||_{\Gamma}||\psi||_{H_{\varepsilon}^1(\Omega; \, \mathcal{T})}.
\end{aligned}
\end{equation}
In other words, the decomposition of $\nabla v_0$ allows us to bound $||\nabla v_0||^2$ in terms of the jump norms $||\timejump{v}||_{\Gamma}$ and $||\timejump{\bm{\tau}}||_{\Gamma}$ along with a few extra additional terms which will themselves need bounding. Bounding these additional terms is the focus of the next two lemmas.

\begin{lemma}\label{psibounds}
We have the following upper bounds for $\psi$:
\begin{equation}
\begin{aligned}\notag
|||\psi|||^2 := \varepsilon||\nabla \psi||^2 + ||\sqrt{\nabla \bigcdot {\bf a} / 2} \,\,\psi||^2 & \leq \varepsilon^{-1}||\nabla v_0||^2,\\
||\psi||^2 & \leq \varepsilon^{-1}C_{\varepsilon}||\nabla v_0||^2,\\
||\psi||^2_{H_{\varepsilon}^1(\Omega; \, \mathcal{T})} & \leq (\varepsilon^{-1}+\varepsilon^{-2}C_{\varepsilon})\, ||\nabla v_0||^2,
\end{aligned}
\end{equation}
where $\displaystyle C_{\varepsilon} := \frac{2}{2C_P^{-2}\varepsilon + \min\limits_{x \in \bar{\Omega}} \nabla \bigcdot {\bf a}(x)}$ and $C_P$ is the Poincar{\'e}-Friedrichs constant of Theorem \ref{pf}.
\end{lemma}
\begin{proof}
Testing \eqref{helm} against $\nabla\psi$ and integrating over $\Omega$ yields
$$\varepsilon ||\nabla \psi||^2 - \int_{\Omega} {\bf a}\psi \bigcdot \nabla \psi \dif x= \int_{\Omega} \nabla v_0 \bigcdot \nabla \psi \dif x.$$
Integration by parts implies that
$$- \int_{\Omega} {\bf a}\psi \bigcdot \nabla \psi \dif x= ||\sqrt{\nabla \bigcdot {\bf a} / 2} \,\,\psi||^2.$$
Using the Cauchy-Schwarz inequality thus gives
$$|||\psi|||^2 \leq ||\nabla v_0||||\nabla\psi||,$$
which implies
$$||\nabla\psi|| \leq \varepsilon^{-1}||\nabla v_0||,$$
which upon substitution yields the first result claimed. Next, notice that the classical Poincar{\'e}-Friedrichs inequality along with the first bound implies that
$$\left|\left|\sqrt{C_P^{-2}\varepsilon + \nabla \bigcdot {\bf a} / 2} \,\,\psi \right|\right|^2 \leq \varepsilon^{-1}||\nabla v_0||^2.$$
Therefore, with the aid of H{\"o}lder's inequality, we have
$$||\psi||^2 = \left|\left|\frac{\sqrt{C_P^{-2}\varepsilon + \nabla \bigcdot {\bf a} / 2} \,\,\psi}{\sqrt{C_P^{-2}\varepsilon + \nabla \bigcdot {\bf a} / 2}}\right|\right|^2 \leq \left|\left|(C_P^{-2}\varepsilon + \nabla \bigcdot {\bf a} / 2)^{-1}\right|\right|_{L^{\infty}(\Omega)}\left|\left|\sqrt{C_P^{-2}\varepsilon + \nabla \bigcdot {\bf a} / 2} \,\,\psi \right|\right|^2.$$
Noting that $\left|\left|(C_P^{-2}\varepsilon + \nabla \bigcdot {\bf a} / 2)^{-1}\right|\right|_{L^{\infty}(\Omega)} = C_{\varepsilon}$ completes proof of the second bound. Finally, recalling that
$$||\psi||^2_{H_{\varepsilon}^1(\Omega; \, \mathcal{T})} =  \varepsilon^{-1}||\psi||^2 + \varepsilon||\nabla\psi||^2,$$
we see that the third bound follows immediately from the first and second bounds.
\end{proof}

\begin{lemma}\label{curlzbound}
We have the following upper bound for $\nabla \!\times \!z$:
$$||\nabla \!\times \!z|| \leq  (1+||{\bf a}||_{L^\infty(\Omega)}\varepsilon^{-1/2}\!\!\sqrt{C_{\varepsilon}})||\nabla v_0||.$$
\end{lemma}
\begin{proof}
Testing \eqref{helm} against $\nabla \!\times \!z$  and integrating over $\Omega$ yields
$$||\nabla \!\times \!z||^2 = (\nabla v_0, \nabla \!\times \!z) + ({\bf a}\psi,\nabla \!\times \!z).$$
Using the Cauchy-Schwarz inequality and H{\"o}lder's inequality we obtain
$$||\nabla \!\times \!z|| \leq ||\nabla v_0||+||{\bf a}||_{L^\infty(\Omega)}||\psi||.$$
The result then follows from the $L^2$ bound for $\psi$ of Lemma \ref{psibounds}.
\end{proof}

In light of \eqref{gradv0bound} and with these two intermediate results in place, we are now ready to state and prove our upper bound for $||\bm{v}_0||_V$.

\begin{theorem}\label{v0bound}
We have the following upper bound for $\bm{v}_0$:
$$||\bm{v}_0||^2_V \leq 2\!\left\{1+C_P^2+(1+2C_P^2)(1+||{\bf a}||_{L^\infty(\Omega)}\varepsilon^{-1/2}\!\!\sqrt{C_{\varepsilon}})^2 \right\}\!||\timejump{v}||^2_{\Gamma} + 2(1+2C_P^2)(1+\varepsilon^{-1}C_{\varepsilon})\, ||\timejump{\bm{\tau}}||^2_{\Gamma},$$
where $\displaystyle C_{\varepsilon} = \frac{2}{2C_P^{-2}\varepsilon + \min\limits_{x \in \bar{\Omega}} \nabla \bigcdot {\bf a}(x)}$ and $C_P$ is the Poincar{\'e}-Friedrichs constant of Theorem \ref{pf}.
\end{theorem}
\begin{proof}
Substituting the results of Lemma \ref{psibounds} and Lemma \ref{curlzbound} into \eqref{gradv0bound} yields
\begin{equation}
\begin{aligned}\notag
||\nabla v_0||^2 \leq (\varepsilon^{-1/2}+||{\bf a}||_{L^\infty(\Omega)}\varepsilon^{-1}\!\!\sqrt{C_{\varepsilon}})||\timejump{v}||_{\Gamma}||\nabla v_0|| + \sqrt{\varepsilon^{-1}+\varepsilon^{-2}C_{\varepsilon}}\, ||\timejump{\bm{\tau}}||_{\Gamma}||\nabla v_0||.
\end{aligned}
\end{equation}
Cancelling terms, using the trivial inequality $2ab \leq a^2 + b^2$ and multiplying by $\varepsilon$ yields
\begin{equation}
\begin{aligned}\notag
\varepsilon||\nabla v_0||^2 \leq 2(1+||{\bf a}||_{L^\infty(\Omega)}\varepsilon^{-1/2}\!\!\sqrt{C_{\varepsilon}})^2||\timejump{v}||^2_{\Gamma} + 2(1+\varepsilon^{-1}C_{\varepsilon})\, ||\timejump{\bm{\tau}}||^2_{\Gamma}.
\end{aligned}
\end{equation}
The discrete Poincar{\'e}-Friedrichs inequality (Theorem \ref{brokenpf}) gives us control over $\varepsilon||v_0||^2$, viz., 
\begin{equation}
\begin{aligned}\notag
\varepsilon||v_0||^2 & \leq 2C^2_P\varepsilon ||\nabla v_0||^2+ 2(1+C_P^2)||\timejump{v}||^2_{\Gamma} \\
&\leq 2\!\left\{1+C_P^2+2C_P^2(1+||{\bf a}||_{L^\infty(\Omega)}\varepsilon^{-1/2}\!\!\sqrt{C_{\varepsilon}})^2 \right\}\!||\timejump{v}||^2_{\Gamma}+ 4C_P^2(1+\varepsilon^{-1}C_{\varepsilon})\, ||\timejump{\bm{\tau}}||^2_{\Gamma}.
\end{aligned}
\end{equation}
The result then follows trivially by combining the bounds for $\varepsilon||v_0||^2$ and $\varepsilon||\nabla v_0||^2$.
\end{proof}

With the upper bounds for $||\bm{v}_0||_V$ and $||\bm{v}_1||_V$ in place, we can now bound $||\bm{v}||_V$ itself.
\begin{theorem}\label{Vupperbound}
For any $\bm{v}\in V$ we have
$$||\bm{v}||_V \leq \sqrt{2\varepsilon||\nabla \bigcdot \bm{\tau} - {\bf a} \bigcdot  \nabla v||^2+(5+3C_P^2)\varepsilon^{-1}||\bm{\tau} + \varepsilon \nabla v||^2 + C_L^2(||\timejump{v}||^2_\Gamma + ||\timejump{\bm{\tau}}||^2_\Gamma)},$$
with $C_L := 2\sqrt{\max\{(1+2C_P^2)(1+\varepsilon^{-1}C_{\varepsilon}), 1 + C_P^2 + (1+2C_P^2)(1+||{\bf a}||_{L^{\infty}(\Omega)}\varepsilon^{-1/2}C_{\varepsilon})^2\}}$ where $C_P$ is the Poincar{\'e}-Friedrichs constant of Theorem \ref{pf} and $\displaystyle C_{\varepsilon} = \frac{2}{2C_P^{-2}\varepsilon + \min\limits_{x \in \bar{\Omega}} \nabla \bigcdot {\bf a}(x)}$.
\end{theorem}
\begin{proof}
Since we have the decomposition $\bm{v} = \bm{v}_0 + \bm{v}_1$, we can apply the elementary inequality
$$||\bm{v}||_V \leq \sqrt{2||\bm{v}_0||_V^2 + 2||\bm{v}_1||_V^2},$$
along with Theorem \ref{v1bound} and Theorem \ref{v0bound} and to deduce the result.
\end{proof}

As with the upper bound for $||\cdot||_E$; upon having deduced an upper bound for $||\cdot||_V$ in Theorem \ref{Vupperbound}, proving a lower bound for $||\cdot||_E$ is now trivial.

\begin{theorem}\label{unormupperbound}
For any $\bm{u} \in U$ we have the lower bound 
\begin{equation}\notag
\begin{aligned}
||\bm{u}||_E & \geq \sqrt{\frac{1}{2\varepsilon}||u||^2 + \frac{\varepsilon}{5+3C_P^2}||\bm{\sigma}||^2 + C_L^{-2}||\hat{u}||_{H_D^{1/2}(\Gamma)}^2 + C_L^{-2}||\hat{\sigma}_n||_{H_N^{-1/2}(\Gamma)}^2 },
\end{aligned}
\end{equation}
with $C_L = 2\sqrt{\max\{(1+2C_P^2)(1+\varepsilon^{-1}C_{\varepsilon}), 1 + C_P^2 + (1+2C_P^2)(1+||{\bf a}||_{L^{\infty}(\Omega)}\varepsilon^{-1/2}C_{\varepsilon})^2\}}$ where $C_P$ is the Poincar{\'e}-Friedrichs constant of Theorem \ref{pf} and $\displaystyle C_{\varepsilon} = \frac{2}{2C_P^{-2}\varepsilon + \min\limits_{x \in \bar{\Omega}} \nabla \bigcdot {\bf a}(x)}$.
\end{theorem}
\begin{proof}
The results of Theorem \ref{Vupperbound} imply that
\begin{equation}\notag
\begin{aligned}
||\bm{u}||_E & = \sup_{\bm{v} \in V} \frac{B(\bm{u},\bm{v})}{||\bm{v}||_V} \\
&\geq \sup_{\bm{v} \in V} \frac{B(\bm{u},\bm{v})}{\sqrt{2\varepsilon||\nabla \bigcdot \bm{\tau} - {\bf a} \bigcdot  \nabla v||^2+(5+3C_P^2)\varepsilon^{-1}||\bm{\tau} + \varepsilon \nabla v||^2 + C_L^2(||\timejump{v}||^2_\Gamma + ||\timejump{\bm{\tau}}||^2_\Gamma)}},
\end{aligned}
\end{equation}
after which an elementary calculation yields the result claimed.
\end{proof}

\subsection{A Robust Norm} We now turn our attention to collecting the results of the two previous subsections in a readable way thus presenting the main result of this paper. To that end, we use the notation $\lesssim$ and $\gtrsim$ to denote inequalities which are true up to some constant independent of $\varepsilon$. Additionally, we will compare our results with those of the other robust norm \cite{C13, CHBTD14} currently in the literature. We begin with a theorem stating what we have proved.
\begin{theorem}\label{maintheorem}
Under the assumption that $||{\bf a}||_{L^{\infty}(\Omega)}$ is order one, we have the robust bounds
\begin{equation}\notag
\begin{aligned}
||\bm{u}||_E &\lesssim \sqrt{{\varepsilon}^{-1}||u||^2 + \varepsilon||\bm{\sigma}||^2 + ||\hat{u}||_{H_D^{1/2}(\Gamma)}^2 + ||\hat{\sigma}_n||_{H_N^{-1/2}(\Gamma)}^2}, \\
||\bm{u}||_E &\gtrsim \sqrt{{\varepsilon}^{-1}||u||^2 + \varepsilon||\bm{\sigma}||^2 + \varepsilon C^{-1}_{\varepsilon}||\hat{u}||_{H_D^{1/2}(\Gamma)}^2 + \varepsilon C^{-1}_{\varepsilon}||\hat{\sigma}_n||_{H_N^{-1/2}(\Gamma)}^2},
\end{aligned}
\end{equation}
where $\displaystyle C^{-1}_{\varepsilon} = C_P^{-2}\varepsilon + (1/2)\min\limits_{x \in \bar{\Omega}} \nabla \bigcdot {\bf a}(x)$ and $C_P$ is the Poincar{\'e}-Friedrichs constant of Theorem \ref{pf}. If $\displaystyle \min_{x \in \bar{\Omega}} \nabla \bigcdot {\bf a}(x) = 0$ then for the lower bound we have
\begin{equation}\notag
\begin{aligned}
||\bm{u}||_E &\gtrsim \sqrt{\varepsilon^{-1}||u||^2 + {\varepsilon}||\bm{\sigma}||^2 + \varepsilon^2||\hat{u}||_{H_D^{1/2}(\Gamma)}^2 + \varepsilon^2||\hat{\sigma}_n||_{H_N^{-1/2}(\Gamma)}^2}.
\end{aligned}
\end{equation}
If, however, $\displaystyle \min_{x \in \bar{\Omega}} \nabla \bigcdot {\bf a}(x) > 0$ and is order one then we have the improved lower bound
\begin{equation}\notag
\begin{aligned}
||\bm{u}||_E &\gtrsim \sqrt{\varepsilon^{-1}||u||^2 + {\varepsilon}||\bm{\sigma}||^2 + \varepsilon||\hat{u}||_{H_D^{1/2}(\Gamma)}^2 + \varepsilon||\hat{\sigma}_n||_{H_N^{-1/2}(\Gamma)}^2 }.
\end{aligned}
\end{equation}
\end{theorem}
\begin{proof}
The first two bounds are a simple consequence of Theorem \ref{unormlowerbound} and Theorem \ref{unormupperbound}. For the specific scaling on the lower bounds, it suffices to notice that if $\displaystyle \min_{x \in \bar{\Omega}} \nabla \bigcdot {\bf a}(x) = 0$ then $C_{\varepsilon}^{-1}$ is of order $\varepsilon$ whereas if $\displaystyle \min_{x \in \bar{\Omega}} \nabla \bigcdot {\bf a}(x) > 0$ and is order one then $C_{\varepsilon}^{-1}$ is also of order one.
\end{proof}

We now compare our results with those of the other major work in this area \cite{C13, CHBTD14}. To do this, we need some notation so that we can introduce their result. Firstly, we recall that the test norm $||\cdot||_{V, \, \text{MD}}$ under consideration in \cite{C13, CHBTD14} is given by
$$||\bm{v}||_{V, \, \text{MD}} := \sqrt{||C_v v||^2 + \varepsilon||\nabla v||^2 + ||{\bf a} \bigcdot \nabla v||^2 + ||C_{\bm{\tau}} \bm{\tau}||^2 + ||\nabla \bigcdot \bm{\tau}||^2},$$
where $\displaystyle C_v |_K := \min\{\sqrt{\varepsilon/|K|}, 1\}$, $C_{\bm{\tau}} |_K := \min\{1/\sqrt{\varepsilon}, 1/\sqrt{|K|}\}$, $K \in \mathcal{T}$. As is the case for our proposed quasi-optimal test norm 
 $$||\bm{v}||_{V}  = \sqrt{\varepsilon||\nabla \bigcdot \bm{\tau} - {\bf a} \bigcdot  \nabla v||^2+\varepsilon^{-1}||\bm{\tau} + \varepsilon \nabla v||^2 + \varepsilon||v||^2 + \varepsilon||\nabla v||^2},$$
  the mesh-dependent norm of \cite{C13, CHBTD14} also has an associated energy norm $||\cdot||_{\widetilde{E}}$ on $U$ given by
$$||\bm{u}||_{\widetilde{E}} := \sup_{\bm{v} \in V} \frac{\widetilde{B}(\bm{u},\bm{v})}{||\bm{v}||_{V, \, \text{MD}}}.$$

We do remark that the bilinear form $\widetilde{B}(\cdot,\cdot)$ used in \cite{C13, CHBTD14} is different from the one used here \eqref{bilinearform}; roughly speaking, we have the approximate relation $\bm{\sigma} \approx \nabla u$ whereas they use $\bm{\sigma} \approx \varepsilon \nabla u$; nevertheless, this does not affect our ability to compare the two different test norms in an abstract way. Indeed, what matters is not the discretization itself but the $\varepsilon$-ratio from the coefficients in front of the upper and lower bounds between the norms, cf., Theorem \ref{DPGconvg}.

For the mesh-dependent test norm $||\cdot||_{V, \, \text{MD}}$, the following result holds (Lemma 1 from \cite{CHBTD14}) which is the analogue of Theorem \ref{maintheorem} for the norm $||\cdot||_V$:
\begin{theorem}\label{channormresult}
For any $\bm{u} \in U$, it holds that
\begin{equation}\notag
\begin{aligned}
||\bm{u}||_{\widetilde{E}} & \lesssim ||u|| + \varepsilon^{-1}||C_{\tau}^{-1}\bm{\sigma}|| + \varepsilon^{-1/2}||\hat{u}||_{H^{1/2}_D(\Gamma)}+\varepsilon^{-1/2}||\hat{\sigma}_n||_{H_N^{-1/2}(\Gamma)}, \\
||\bm{u}||_{\widetilde{E}} & \gtrsim ||u|| + ||\bm{\sigma}|| + \varepsilon||\hat{u}||_{H^{1/2}_D(\Gamma)}+\sqrt{\varepsilon}||\hat{\sigma}_n||_{H_N^{-1/2}(\Gamma)},
\end{aligned}
\end{equation}
provided that $\nabla \bigcdot {\bf a} = 0$, $\nabla \!\times\! {\bf a} = \bm{0}$, $\nabla {\bf a} + \nabla {\bf a^T} - (\nabla \bigcdot {\bf a}) \,I_d$ is order one and that there exists a positive order one constant $C_{\bf{a}}$ such that
$C_{\bf a} \leq |{\bf a}|^2$.
\end{theorem}
Before we dive into a direct comparison between the two proposed test norms, we note that the trace norms appearing in this theorem are not the same as those in Theorem \ref{maintheorem}; instead, they are induced by the standard broken ($\varepsilon$-free) $H^1$ and $H(\text{div})$ norms. 

Firstly, we remark that (as has been proved so far) the mesh-dependent test norm $||\cdot||_{V, \, \text{MD}}$ of \cite{CHBTD14} is deficient by a factor of $\varepsilon^{-1/2}$ in the $\bm{\sigma}$ component, $\varepsilon^{-3/2}$ in the $\hat{u}$ component and $\varepsilon^{-1}$ in the $\hat{\sigma}_n$ component whereas our proposed test norm $||\cdot||_V$ is not deficient in the $\bm{\sigma}$ component and is deficient by a factor of $\varepsilon^{-1}$ in both trace components thus our proposed test norm $||\cdot ||_V$ has tighter {\bf proved} bounds than those of the mesh-dependent test norm $||\cdot||_{V, \, \text{MD}}$. In addition to these proved tighter bounds, our proposed test norm $||\cdot||_V$ requires less assumptions in order to achieve these bounds; indeed, we require only that $||{\bf a}||_{L^{\infty}(\Omega)}$ exist and be order one whereas the mesh-dependent test norm $||\cdot||_{V, \, \text{MD}}$ has additional restrictions placed on the convection in order for the bounds of Theorem \ref{channormresult} to hold. Finally, we remark that under the mild additional assumption that $\displaystyle \min_{x \in \bar{\Omega}} \nabla \bigcdot {\bf a}(x) > 0$, our trace norm deficiency can be improved to a factor of only $\varepsilon^{-1/2}$ which may very well be optimal.

At this point, we must stress that what we are not saying is that the proposed test norm $||\cdot||_V$ is better than the mesh-dependent test norm $||\cdot||_{V, \, \text{MD}}$ of \cite{C13, CHBTD14}, only that tighter bounds for $||\cdot||_V$ have been proved under less restrictive conditions on the convection. Indeed, it is worth noting that the authors in \cite{CHBTD14} state that the mesh-dependent test norm $||\cdot||_{V, \, \text{MD}}$ seems to work well under much less restrictive conditions than those required in Theorem \ref{channormresult}. In order to facilitate a comparison, we will apply the two different test norms to a variety of different problems and compare the results in the numerical experiments section.

\section{Numerical Experiments}

Before we begin our numerical experiments, we first need to select our discrete trial space $U_h$. Recall that $u_h \in U_h$, the solution to \eqref{discreteproblem}, satisfies Theorem \ref{DPGconvg} which tells us that choosing the right DPG trial space $U_h$ is just as crucial as selecting a good optimal test norm. In particular, the finite element spaces need to be chosen such that for any polynomial degree $p > 0$ we have
\begin{equation}
\begin{aligned}\label{optimalorder}
\inf_{\bm{w}_h \in U_h} ||\bm{u} - \bm{w}_h||_U = \mathcal{O}(h^{p+1}),
\end{aligned}
\end{equation}
where $\displaystyle h := \max_{K \in \mathcal{T}} \,\text{diam}(K)$ is the \emph{maximum mesh-size}. We thus choose our finite element space to be $U_h := S_h^p \!\times\! \bm{S}_h^p \!\times\! Q_h^{p+1} \!\times\! R_h^{p+1}$ where 
\begin{equation}
\begin{aligned}\notag
S_h^p &:= \{q: \Omega \to \mathbb{R} \,|\, q|_K \in \mathcal{P}^p(K)\},\\
\bm{S}_h^p &:=[S_h^p]^d, \\
Q_h^{p+1} & := \{q: \Gamma \to \mathbb{R} \,|\, \exists w \in S_h^{p} \cap H^1_0(\Omega) \text{ such that } q = w|_{\Gamma}\}, \\
R_h^{p+1} &:= \{q: \Gamma \to \mathbb{R} \,|\, q |_E \in \mathcal{P}^{p+1}(E)\},
\end{aligned}
\end{equation}
with $P^p(U)$ denoting the space of polynomials of degree $p$ in each variable on the open set $U$. Using these finite element spaces, we do indeed have \eqref{optimalorder} provided that the exact solution $\bm{u}$ to \eqref{ultraweakshort} has sufficient regularity.

To refine the mesh, we also need an \emph{a posteriori} error estimator; fortunately, under certain conditions, it has been shown \cite{CDG14} that DPG has just such an error estimator available for general problems of the form \eqref{contproblem}. Firstly, we must find the discrete Riesz representative $\tilde{\xi}_h \in \widetilde{V}_h$ of the residual which satisfies
\begin{equation}
\begin{aligned}\label{aposteriori}
(\tilde{\xi}_h, \tilde{v}_h)_V = l(\tilde{v}_h) - B(u_h, \tilde{v}_h) \qquad \forall \tilde{v}_h \in \widetilde{V}_h,
\end{aligned}
\end{equation}
where $\widetilde{V}_h$ is the enriched test space and $(\cdot,\cdot)_V$ is the localizable inner-product which induces the test norm $||\cdot||_V$. Obviously since the enriched test space is broken, the global problem \eqref{aposteriori} may be reduced to a local problem on each element. After it has been computed, $||\tilde{\xi}_h||_V$ provides an \emph{a posteriori} error bound for the error $||{u}-{u}_h||_U$ and $||\tilde{\xi}_h \,|_K||_V$ may be used as a refinement indicator for the element $K$. In all of our numerical experiments, we start from a coarse grid and allow the indicator $||\tilde{\xi}_h \,|_K||_V$ to drive adaptivity; specifically, the top 10\% of all elements as ranked by the indicator are refined on each cycle.

During the course of running the numerical experiments, we observed that the norm
$$||\bm{v}||_{V}  = \sqrt{\varepsilon||\nabla \bigcdot \bm{\tau} - {\bf a} \bigcdot  \nabla v||^2+\varepsilon^{-1}||\bm{\tau} + \varepsilon \nabla v||^2 + \varepsilon||v||^2 + \varepsilon||\nabla v||^2},$$
suffers from the same problems that plague the  standard quasi-optimal test norm {\bf --} we end up with refinement around the inflow boundary of the domain despite there being no boundary layers present there. We fix this in the same way as \cite{C13, CHBTD14}; namely, we introduce a mesh-dependent parameter $\displaystyle C_{\bm{\tau}} |_K := \min\{1/\sqrt{\varepsilon}, 1/\sqrt{|K|}\}$, $K \in \mathcal{T}$.  When attached to the term corresponding to the gradient $\bm{\sigma}$, viz.,
$$||\bm{v}||_{V}  = \sqrt{\varepsilon||\nabla \bigcdot \bm{\tau} - {\bf a} \bigcdot  \nabla v||^2+||C_{\bm{\tau}}(\bm{\tau} + \varepsilon \nabla v)||^2 + \varepsilon||v||^2 + \varepsilon||\nabla v||^2},$$
this issue seems to disappear for moderately small $\varepsilon$. Of course, the analysis conducted in the previous section is no longer valid for the new mesh-dependent test norm. However, a careful analysis reveals that not much changes; indeed, we see that only robustness of the gradient term $\bm{\sigma}$ is affected by a factor of $\displaystyle \max_{K \in \mathcal{T}} \,C_{\bm{\tau}}^{-2}/\varepsilon$ but as $h \to 0^+$, $\displaystyle \max_{K \in \mathcal{T}} \,C_{\bm{\tau}}^{-2}/\varepsilon \to 1^+$ and hence robustness is restored once the mesh is sufficiently refined.

We consider three different numerical experiments. Example 1 will focus on ensuring the order of convergence and that the proposed test norm performs well when used to drive mesh adaption. The second two examples will be used to compare our test norm with the test norm  
\begin{equation}
\begin{aligned}\notag
||\bm{v}||_{V,\, \text{MD}} = \sqrt{||C_v v||^2 + \varepsilon||\nabla v||^2 + ||{\bf a} \bigcdot \nabla v||^2 + ||C_{\bm{\tau}} \bm{\tau}||^2 + ||\nabla \bigcdot \bm{\tau}||^2},
\end{aligned}
\end{equation}
from \cite{C13, CHBTD14}.

All of the simulations in this paper are based on the deal.II finite element library \cite{BHK07}.

\subsection{Example 1}

For our first example, we consider a problem from \cite{SZ09} which is also considered as a time-dependent variant in  \cite{CGM14}. We set $\Omega = (0,1)^2$, $u|_{\partial \Omega} = 0$, ${\bf a} = (1,1)^T$ and $f$ is chosen such that the solution to problem \eqref{model_primal} is given by
$$u(x,y) = \bigg(\frac{e^{(x-1)/\varepsilon}-1}{e^{-1/\varepsilon}-1}+x-1\bigg)\!\bigg(\frac{e^{(y-1)/\varepsilon}-1}{e^{-1/\varepsilon}-1}+y-1 \bigg).$$
The solution has boundary layers of width $\mathcal{O}(\varepsilon)$ around the outflow boundary of the domain.

To begin, we run the adaptive algorithm driven by the \emph{a posteriori} error estimator to around one million degrees of freedom for the diffusion coefficients $\varepsilon = 1, \,10^{-2}, \,10^{-3},\, 10^{-4}$ and the polynomial degrees $p = 0,\, 1, \,2,\, 3$. We observe that the estimator $\eta$, the $L^2$ solution error $||u - u_h||$ and the $\varepsilon$-scaled $L^2$ gradient error $\varepsilon||\bm{\sigma} - \bm{\sigma}_h||$ all converge with optimal order once the boundary has been sufficiently refined although we were unable to confirm this for $\varepsilon = 10^{-4}$ due to reaching the asymptotic refinement regime only near the end of the computation. The convergence results for $\varepsilon = 0.01$ and the various different polynomial degrees are displayed in Figure \ref{example1convergence}.

Next, we plot the ratio of the $L^2$ solution error $||u - u_h||$ to the $\varepsilon$-scaled $L^2$ gradient error $\varepsilon||\bm{\sigma} - \bm{\sigma}_h||$ in Figure \ref{example1ratios} for $p = 3$ and the different values of $\varepsilon$. Given that the scaling on these norms is the same relative scaling present in the trial norm \eqref{unorm} and that our test norm is designed for robustness in this trial norm, we expect that these two errors should be linked independently of $\varepsilon$. We do indeed observe this asymptotically; in fact, the ratio is an almost perfect value of one for $\varepsilon = 10^{-2}$ and  $\varepsilon = 10^{-3}$. In the preasymptotic regime, this ratio is large but the results are unreliable as the exact solution $u$ contains boundary layers meaning the error values, computed using quadrature, will not be accurate until the boundary layers are sufficiently refined. 

Although this is not the focus of the paper, for interest, we also plot the ratio of the error estimator $\eta$ to the $L^2$ solution error $||u - u_h||$  for $p = 3$ and the different $\varepsilon$ values in Figure \ref{example1ratios}. The results show that once the boundary has been sufficiently refined this ratio is one for $\varepsilon = 10^{-2}$ and $\varepsilon = 10^{-3}$  indicating asymptotic parity of the $L^2$ solution error and the estimator for this problem.

\begin{figure}[h]
\centering
\includegraphics[scale=0.262]{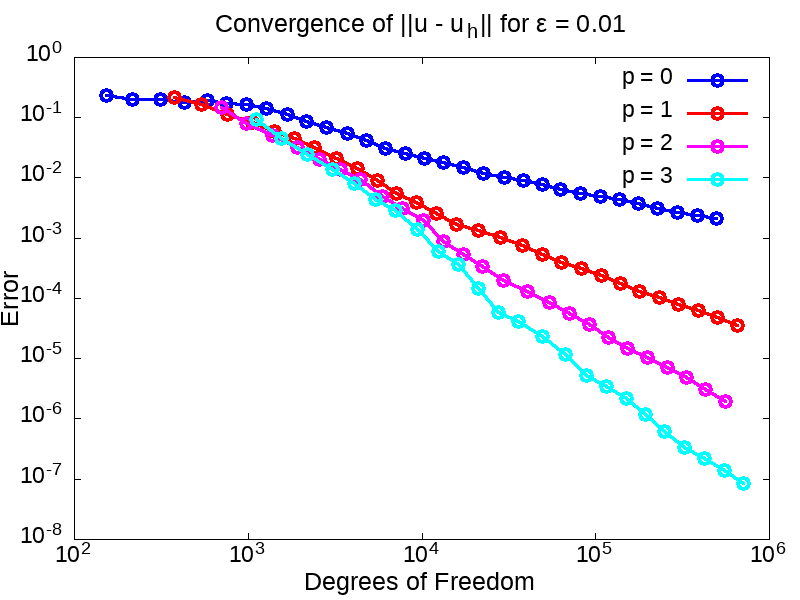} \includegraphics[scale=0.262]{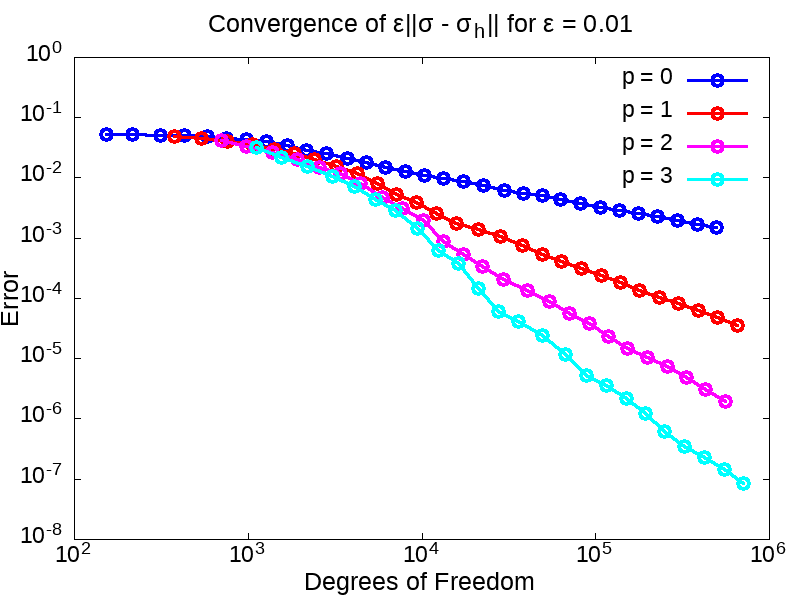}
\caption{Example 1: Convergence results for $\varepsilon = 0.01$.}
\label{example1convergence}
\end{figure}

\begin{figure}[h]
\centering
\includegraphics[scale=0.262]{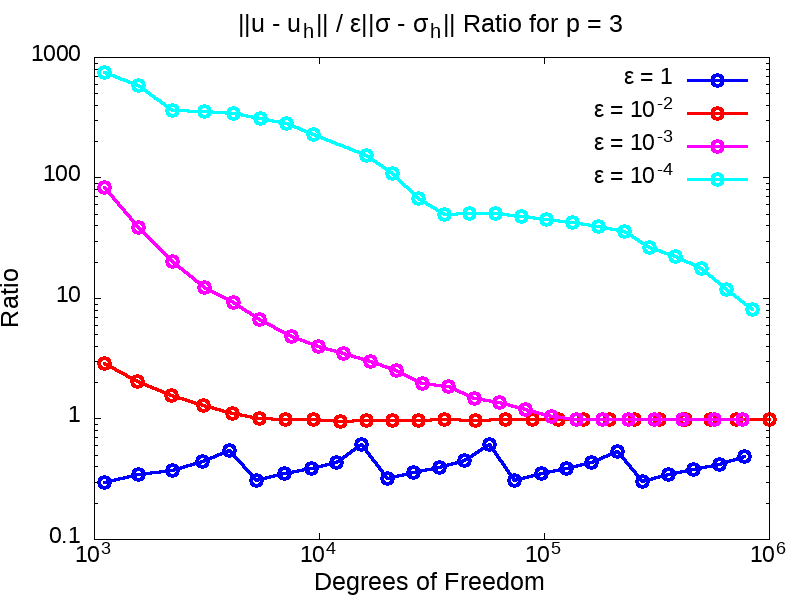} \includegraphics[scale=0.262]{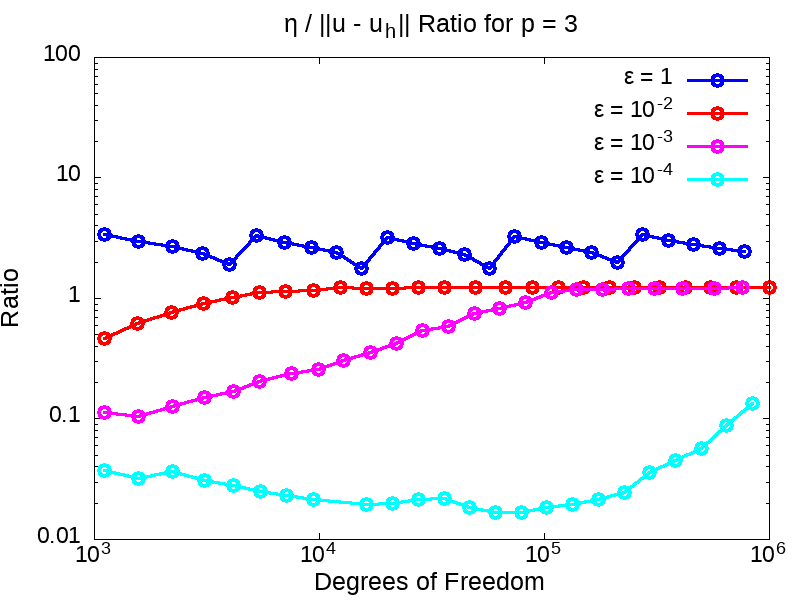}
\caption{Example 1: Error ratios for $p = 3$.}
\label{example1ratios}
\end{figure}
\begin{figure}[h]
\centering
\includegraphics[scale=0.325]{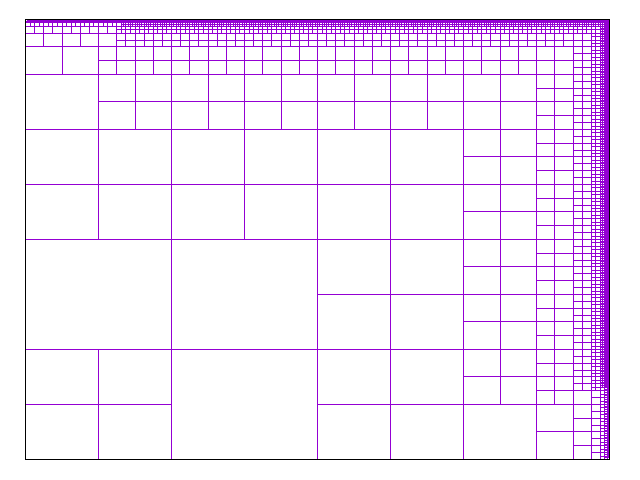} \includegraphics[scale=0.325]{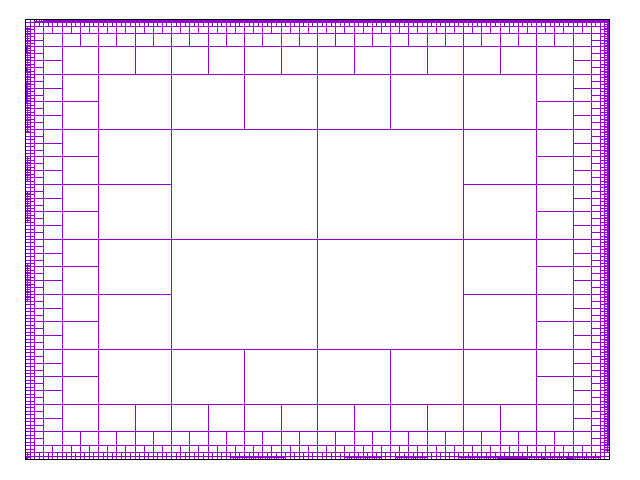}
\caption{Example 1: Meshes at final runtime for $\varepsilon = 10^{-3}$ (left) and $\varepsilon = 10^{-4}$ (right).}
\label{example1grids}
\end{figure}

We display the grids from final runtime for $p = 3$ and $\varepsilon = 10^{-3}, \, 10^{-4}$ in Figure \ref{example1grids}. The results show that our proposed test norm performs well for $\varepsilon = 10^{-3}$ but for $\varepsilon = 10^{-4}$ even the mesh-dependent modification made to the test norm is insufficient to stop unnecessary refinement around the inflow boundary in the pre-asymptotic regime despite no boundary layers being present.

\subsection{Example 2}
For this example, we consider the so-called Eriksson-Johnson model problem \cite{EJ93} thus we set $\Omega = (0,1)^2$, ${\bf a} = (1,0)^T$, $f = 0$ and consider the following boundary conditions:
\begin{equation}
\begin{aligned}\notag
u(0,y) & = u_0(y),  \quad\qquad && y \in (0,1), \\
u(1,y) & = 0, && y \in (0,1), \\
({\bf a}u - \varepsilon \nabla u) \bigcdot \bm{n}(x,0) & =  0, \qquad && x \in (0,1),\\
({\bf a}u - \varepsilon \nabla u) \bigcdot \bm{n}(x,1) & = 0, \qquad && x \in (0,1).
\end{aligned}
\end{equation}
The solution to this problem has a boundary layer of width $\mathcal{O}(\varepsilon)$ at the outflow boundary. Using separation of variables, this problem has the exact solution
\begin{equation}
\begin{aligned}\notag
u(x,y) = \sum_{n=0}^{\infty} C_n \frac{\exp(r_2(x-1)) - \exp(r_1(x-1))}{\exp(-r_2)-\exp(-r_1)}\cos(n\pi y),
\end{aligned}
\end{equation}
where
\begin{equation}
\begin{aligned}\notag
r_{1,\,2} = \frac{1 \pm \sqrt{1 + 4\varepsilon \lambda_n}}{2 \varepsilon}, \qquad\qquad\qquad \lambda_n = n^2 \pi^2 \varepsilon, 
\end{aligned}
\end{equation}
and
\begin{equation*}
C_n = \begin{cases}\displaystyle \int_0^1 u_0(y) \dif y \qquad & \text{if } n = 0,
 \\ \displaystyle 2\int_0^1 u_0(y)\cos(n \pi y) \dif y \qquad & \text{otherwise.} \end{cases}
\end{equation*}

\begin{figure}[h]
\centering
\includegraphics[scale=0.382]{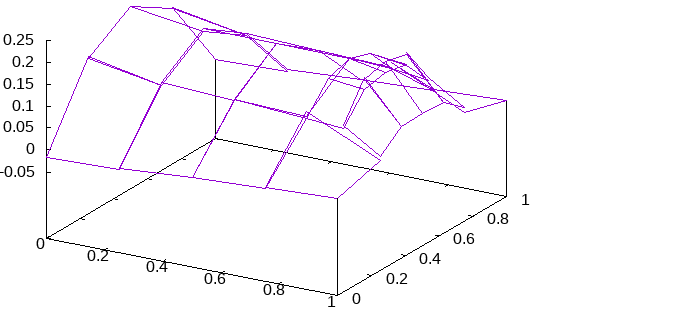}\hspace{-18mm} \includegraphics[scale=0.382]{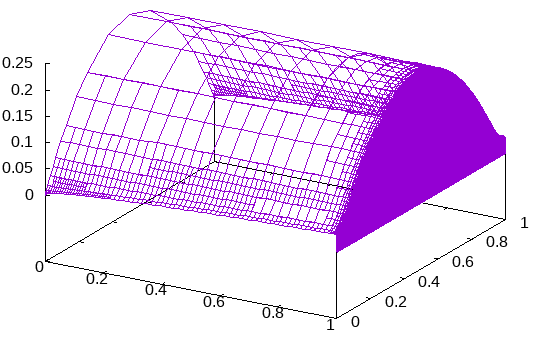}
\caption{Example 2: Numerical solution for $\varepsilon = 10^{-3}$ in the pre-asymptotic (left) and asymptotic (right) mesh refinement regimes.}
\label{example2solutions}
\end{figure}
\vspace{-4mm}
\begin{figure}[h]
\centering
\includegraphics[scale=0.325]{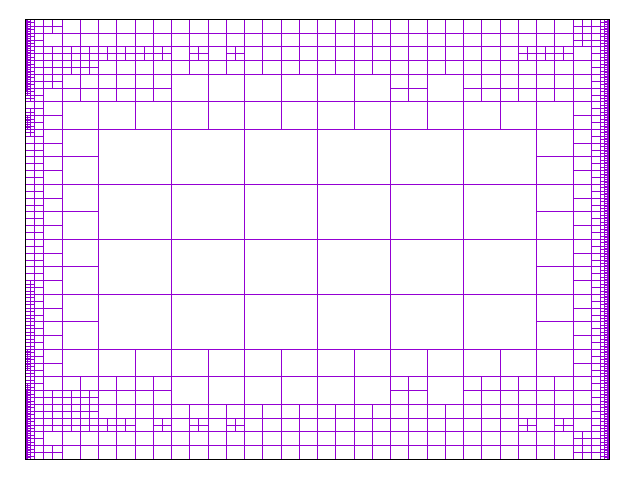} \includegraphics[scale=0.325]{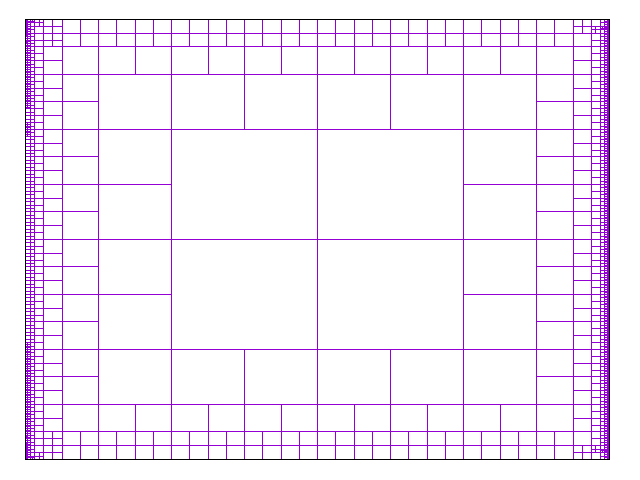}
\caption{Example 2: Meshes at around one million degrees of freedom for $\varepsilon = 10^{-4}$ under the test norm $||\cdot||_{V, \,\text{MD}}$ (left) and our test norm $||\cdot||_V$ (right).}
\label{example2grids}
\end{figure}

We will use the function $u_0(y) = y(1-y)$ for the inflow boundary condition. The numerical solution for $\varepsilon = 10^{-3}$ and $p = 3$ in the pre-asymptotic and asymptotic mesh refinement regimes is plotted in Figure \ref{example2solutions}. We see that the numerical solution is very stable even with barely any elements present and that, again, the estimator using our proposed test norm correctly picks up and refines the boundary layer.

Next, we compare results for the adaptive algorithm as driven by our test norm $||\cdot||_V$ versus the test norm $||\cdot||_{V,\, \text{MD}}$ from \cite{C13, CHBTD14}; in each case, adaptivity is driven by the respective test norm. For $\varepsilon = 10^{-4}$, we observe degeneration in mesh quality, plotted in Figure \ref{example2grids} (for $p$ = 3), for \emph{both} test norms with extraneous refinement around the inflow boundary despite no layers present. 

The $L^2$ solution errors $||u - u_h||$ for $\varepsilon = 10^{-2}$ and $\varepsilon = 10^{-3}$, $p = 2$ and $p = 3$ under the two different test norms $||\cdot||_V$ and $||\cdot||_{V,\, \text{MD}}$ are shown in Figure \ref{example2convergence}. The results show that our test norm performs slightly better for $\varepsilon = 10^{-2}$ while the test norm $||\cdot||_{V,\, \text{MD}}$ performs better for $\varepsilon = 10^{-3}$ although the difference in both cases is largely negligible indicating that, for this example, the two test norms perform more or less identically with respect to minimizing the $L^2$ solution error.

\begin{figure}[h]
\centering
\includegraphics[scale=0.262]{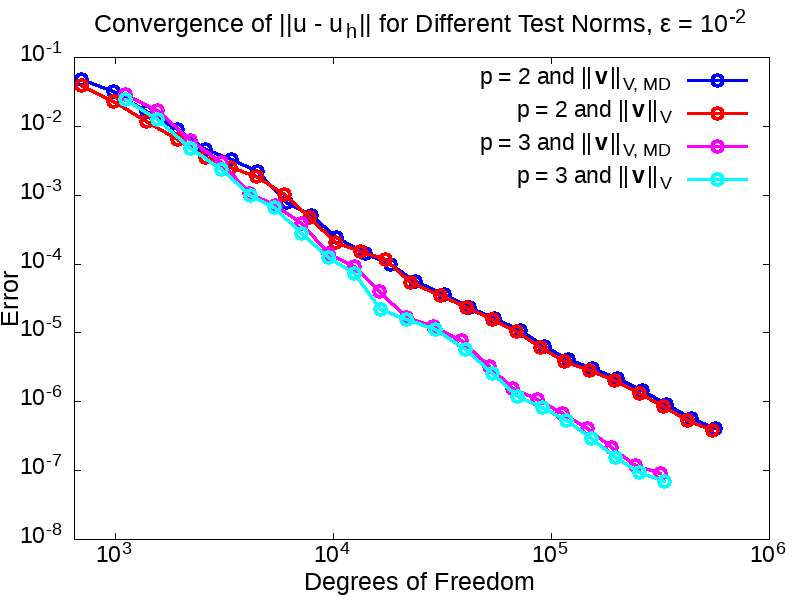} \includegraphics[scale=0.262]{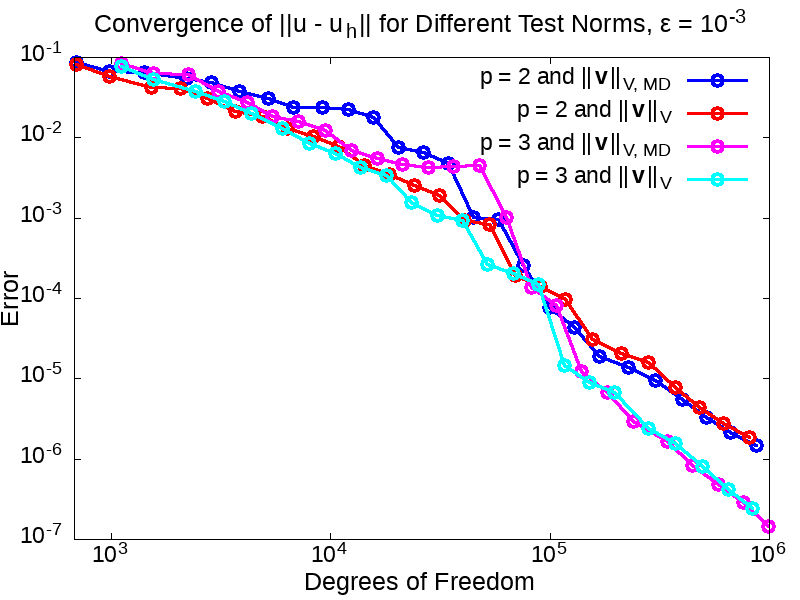}
\caption{Example 2: Convergence of the $L^2$ solution error $||u - u_h||$ under the two different test norms for $\varepsilon = 10^{-2}$ (left) and $\varepsilon = 10^{-3}$ (right).}
\label{example2convergence}
\end{figure}
\vspace{-4mm}
\begin{figure}[h]
\centering
\includegraphics[scale=0.325]{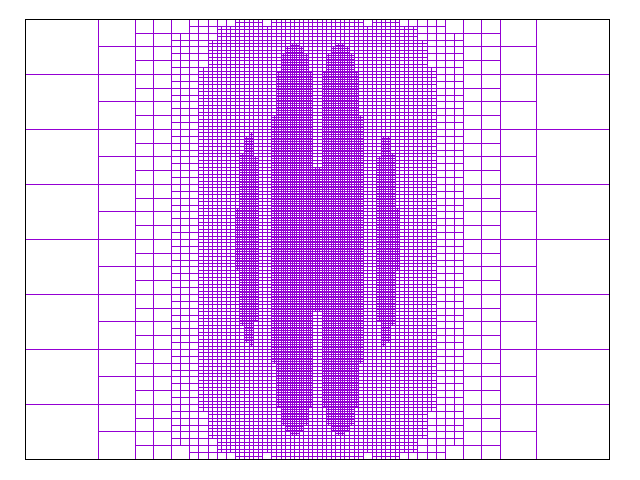} \includegraphics[scale=0.325]{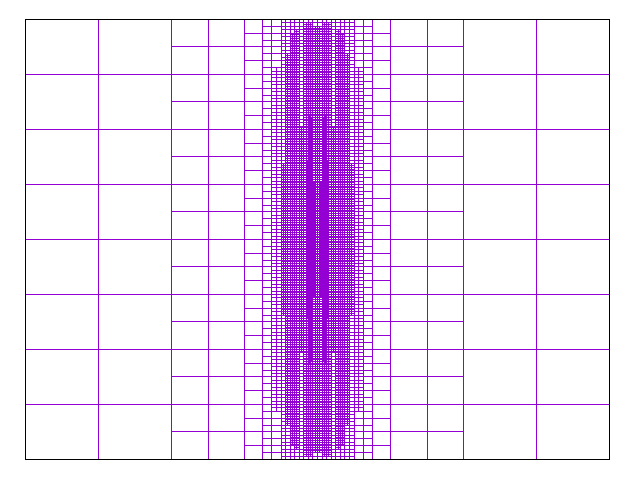}
\caption{Example 3: Meshes at final runtime for $\varepsilon = 10^{-2}$ (left) and $\varepsilon = 10^{-3}$ (right).}
\label{example3grids}
\end{figure}

\subsection{Example 3}

Here, we select another example from \cite{SZ09}. We set $\Omega = (-1,1)^2$ and consider the non-constant convection ${\bf a} = (x,y)^T$. The Dirichlet boundary conditions and right-hand side $f$ are then chosen such that the solution to \eqref{model_primal} is given by
$$u(x,y) = \text{erf}(x/\sqrt{2\varepsilon})(1-y^2).$$
The solution to this problem exhibits an interior layer of width $\mathcal{O}(\sqrt{\varepsilon})$. Meshes for $p = 3$ and $\varepsilon = 10^{-2}, \,10^{-3}$ are displayed in Figure \ref{example3grids} and clearly show that our test norm successfully picks up and refines the interior layer present in the solution.

As in Example 2, we compare results for the adaptive algorithm as driven by our test norm $||\cdot||_V$ versus the test norm $||\cdot||_{V,\, \text{MD}}$ from \cite{C13, CHBTD14}; in each case, adaptivity is driven by the respective test norm. The results given in Figure \ref{example3convergence} show that $L^2$ solution errors  $||u-u_h||$ under both test norms are robust with respect to $\varepsilon$. Our test norm outperforms the test norm $||\cdot||_{V,\, \text{MD}}$ for $\varepsilon = 10^{-2}$ with the situation reversed for $\varepsilon = 10^{-4}$, however, both test norms deliver near identical values for the $L^2$ solution error once the mesh has been sufficiently refined.

\begin{figure}[h]
\centering
\includegraphics[scale=0.262]{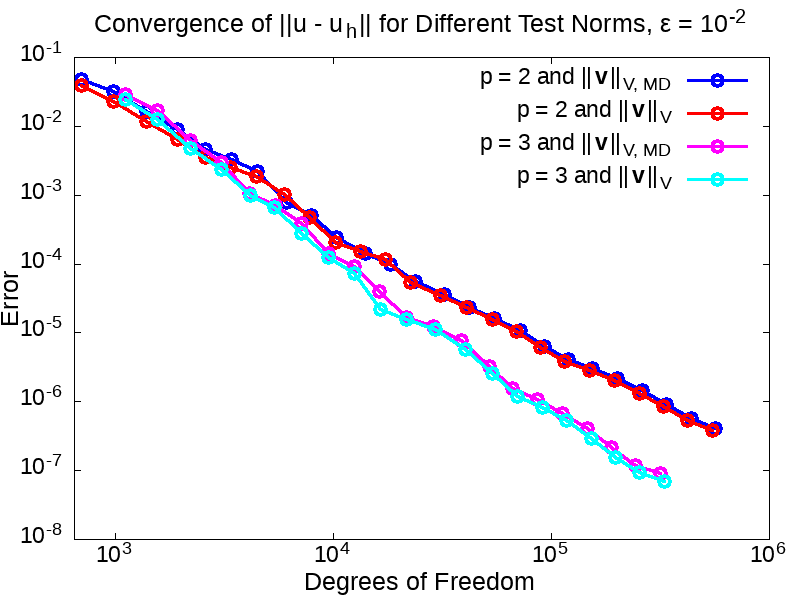} \includegraphics[scale=0.262]{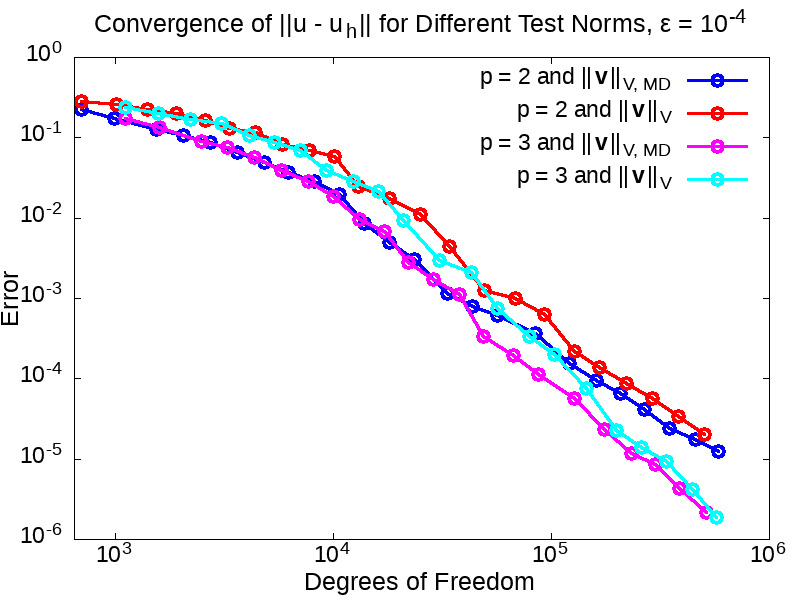}
\caption{Example 3: Convergence of the $L^2$ solution error $||u - u_h||$ under the two different test norms for $\varepsilon = 10^{-2}$ (left) and $\varepsilon = 10^{-4}$ (right).}
\label{example3convergence}
\end{figure}
\vspace{-3mm}
\section{Conclusions}

We proposed the mesh-dependent quasi-optimal test norm
$$||\bm{v}||_{V}  = \sqrt{\varepsilon||\nabla \bigcdot \bm{\tau} - {\bf a} \bigcdot  \nabla v||^2+||C_{\bm{\tau}}(\bm{\tau} + \varepsilon \nabla v)||^2 + \varepsilon||v||^2 + \varepsilon||\nabla v||^2},$$
$\displaystyle C_{\bm{\tau}} |_K = \min\{1/\sqrt{\varepsilon}, 1/\sqrt{|K|}\}$, $K \in \mathcal{T}$ for use in the DPG method based on the ultra-weak formulation of the convection-diffusion equation. We proved that this test norm is robust in the solution component and also robust in the gradient component once the mesh has been sufficiently refined; additionally, the proposed test norm was proven to have favorable scalings in the trace components. The robustness proof requires only minimal assumptions on the convection in contrast to similar results in the literature \cite{C13, CHBTD14}. 

Numerical experiments show that, when compared with the mesh-dependent test norm from \cite{C13, CHBTD14}, our proposed test norm was competitive delivering near identical $L^2$ solution errors and producing similar meshes. The numerics imply that quasi-optimal test norms, when appropriately augmented with mesh-dependent terms and scaled with $\varepsilon$, can be competitive for the DPG method as applied to the convection-diffusion equation. Nevertheless, we believe that there is still much work to be done on this topic as none of the test norms in the literature (including our proposed test norm) performed well in all respects for $\varepsilon \leq \, \approx10^{-4}$.

\bibliographystyle{amsplain}
\bibliography{paper}
\end{document}